\documentclass[11pt]{article}

\usepackage[margin=1in]{geometry}
\usepackage{amsmath,amssymb,amsthm,mathtools}
\usepackage{array}
\usepackage{mathrsfs}
\usepackage{enumitem}
\usepackage{microtype}
\usepackage{xcolor}
\usepackage{hyperref}
\usepackage[nameinlink]{cleveref}

\hypersetup{
  colorlinks=true,
  linkcolor=blue,
  citecolor=red,
  urlcolor=green,
  pdftitle={A Complete Characterization of Sequential Testability and Change Detectability in i.i.d. Models},
  pdfauthor={Aaditya Ramdas},
  pdfkeywords={sequential testing, power-one tests, e-processes, optimal growth, quickest change detection, fixed-horizon false-alarm control, e-detectors, average run length}
}

\newtheorem{theorem}{Theorem}[section]
\newtheorem{proposition}[theorem]{Proposition}
\newtheorem{lemma}[theorem]{Lemma}
\newtheorem{corollary}[theorem]{Corollary}
\theoremstyle{definition}
\newtheorem{definition}[theorem]{Definition}
\newtheorem{remark}[theorem]{Remark}
\newtheorem{example}[theorem]{Example}
\numberwithin{equation}{section}

\newcommand{\Pcal}{\mathcal P}
\newcommand{\Qcal}{\mathcal Q}
\newcommand{\X}{\mathsf X}
\newcommand{\B}{\mathcal B}
\newcommand{\F}{\mathcal F}
\newcommand{\E}{\mathbb E}
\newcommand{\Prob}{\mathbb P}
\newcommand{\M}{\mathcal M}
\newcommand{\KL}{\mathrm{KL}}
\newcommand{\TV}{\mathrm{TV}}
\newcommand{\R}{\mathbb R}
\newcommand{\one}{\mathbf 1}
\newcommand{\Eproc}{\mathcal E}
\newcommand{\Detector}{\mathcal D}
\newcommand{\Sproc}{\mathcal S}

\DeclareMathOperator{\size}{size}
\DeclareMathOperator{\pow}{power}

\title{A complete characterization of sequential testability\\ and change detectability in i.i.d. models}
\author{Aaditya Ramdas\\Department of Statistics, Stanford University\\\texttt{aramdas@stanford.edu}}
\date{\today}

\begin{document}
\maketitle

\begin{abstract}
We give qualitative and quantitative necessary and sufficient conditions for sequential testing and change detection in an i.i.d. composite problem.  A level-\(\alpha\) test with power one against every alternative exists if and only if the alternatives are separated from the null by a countable family of finite-block events.  Equivalent conditions use randomized tests, bounded scores, e-processes, reduced-filtration test supermartingales, and a countable cover with positive finite-block weak-$*$ convex-hull separation.  The construction gives pointwise expected sample size \(O_Q(\log(1/\alpha))\).
Quantitatively, one e-process attains the intrinsic pointwise growth rate \(\Gamma(Q)\) at every alternative if and only if a countable family of bounded block bets recovers these rates.  The same condition characterizes simultaneous attainment of the sharp expected-sample-size coefficient \(1/\Gamma(Q)\); arbitrary attainable lower rate profiles admit an analogous characterization.  Total-variation separability of the alternatives suffices for recovery.  Weak lower semicontinuity of the KL projection, in particular weak compactness of the null, also identifies \(\Gamma(Q)\) with \(\inf_{P\in\mathcal P}\mathrm{KL}(Q\|P)\).
The qualitative and quantitative conditions also characterize change detectability under fixed-time false-alarm bounds \(\Prob_{P^\infty}(T_\gamma\le n)\le n/\gamma\), together with ordinary ARL control.  Whenever countable rate recovery holds, for any \(\eta>0\) one calibrated e-detector has robust Lorden delay asymptotic to \((\log\gamma)/\Gamma(Q)\) at finite positive rates, while satisfying the stronger optional-horizon inequality and placing every null ARL in \([\gamma,(1+\eta)\gamma+1]\).  Examples exhibit a strict cost of simultaneous optimality and show that ordinary ARL alone supports neither characterization, even with finite, tightly calibrated null means.
\end{abstract}


\section{Introduction}\label{sec:introduction}

Sequential tests allow the sample size to be chosen from the data while maintaining a prescribed type-I error probability.  In favorable cases such tests can have \emph{power one}: under every alternative of interest, the test rejects in finite time almost surely.  Classical examples go back to Wald's sequential likelihood-ratio test and to the work of Robbins, Darling, Siegmund, Lai, and collaborators on power-one procedures and confidence sequences \cite{wald1945,darling-robbins-1967,darling-robbins-1968,robbins-siegmund-1970,lai1977}.  Modern anytime-valid inference expresses many such procedures through test martingales, nonnegative supermartingales, or e-processes.

Larsson et al.~\cite{larsson-ruf-ramdas-2026} extended classical work by Le Cam and Kraft to recently derive a complete characterization of (nonsequential) testability: when testing any composite $\Pcal$ against any composite $\Qcal$, they give an assumption-free geometric characterization of when there exists a test with minimax risk strictly less than one (power strictly exceeds the level). This paper asks for (and achieves) an analogous exact existence criterion for sequential testing in the i.i.d. setting. 

Let \(\Pcal\) be a null class of distributions on a measurable sample space \((\X,\B)\), and let \(\Qcal\) be an alternative class, both always assumed nonempty.  We seek a single stopping rule \(\tau\) such that
\[
  \sup_{P\in\Pcal}P^\infty(\tau<\infty)\le \alpha,
  \qquad
  Q^\infty(\tau<\infty)=1\quad\text{for every }Q\in\Qcal.
\]
The main result states that such a test exists if and only if the alternatives can be separated from the null by countably many finite-sample inequalities.  One form of the condition is that there are events \(B_m\in\B^{\otimes n_m}\) such that every \(Q\in\Qcal\) satisfies
\[
  Q^{n_m}(B_m)>\sup_{P\in\Pcal}P^{n_m}(B_m)
\]
for at least one \(m\).  Equivalently, one may use randomized fixed-sample tests whose power exceeds their worst-case null size, or bounded finite-block scores whose expectation is uniformly nonpositive under the null but positive under each alternative for at least one block.

The proof is constructive.  A separating finite-block score gives a blockwise betting factor with null expectation at most one and positive logarithmic growth under the alternative it separates.  A countable mixture over all blocks and rational betting fractions yields an e-process that diverges under every alternative.  Conversely, any power-one sequential test yields a countable family of finite-sample events by taking the rejection events \(\{\tau\le n\}\).

The pointwise strict gaps in the main theorem need not be uniform over the full alternative class.  Nevertheless, splitting alternatives according to rational lower bounds on those gaps yields a countable cover by uniformly separated subfamilies.  The unrestricted fixed-sample minimax identity then gives an equivalent geometric characterization in terms of positive total-variation distance between finite-block weak-$*$ closed convex hulls.  We state this as a corollary immediately after the main theorem, thereby separating the operational content of the theorem from its geometric representation.

The quantitative question asks whether one procedure, without knowing \(Q\), can attain the intrinsic pointwise optimal growth rate \(\Gamma(Q)\) identified by Ram and Ramdas \cite{ram-ramdas-optimal-growth-2026}.  The exact condition strengthens strict separation to countable rate recovery: bounded finite-block e-variables must have normalized expected log payoffs whose supremum is \(\Gamma(Q)\) at every alternative.  We prove that this is equivalent both to simultaneous asymptotic REGROW optimality and to sharp expected sample size \((1+o(1))\log(1/\alpha)/\Gamma(Q)\) at finite positive rates.  More generally, countable block envelopes describe all attainable lower rate profiles.  The intrinsic rate may be smaller than the raw KL projection; even when the two agree, simultaneous attainment can fail.

The same question has a change-detection counterpart.  Suppose the observations are i.i.d. from some \(P\in\Pcal\) before an unknown changepoint and i.i.d. from some \(Q\in\Qcal\) afterward.  We require ordinary ARL at least \(\gamma\) and the fixed-time false-alarm bounds
\[
  \Prob_{P^\infty}(T_\gamma\le n)\le n/\gamma
  \qquad(P\in\Pcal,\ n\ge1).
\]
The finite-block condition is necessary and sufficient for such a family to have pointwise \(O_Q(\log\gamma)\) robust Lorden delay.  The quantitative rate-recovery condition gives the sharp coefficient \(1/\Gamma(Q)\).  The constructive direction satisfies the stronger optional-horizon inequality
\[
  \Prob_{P^\infty}(T_\gamma\le\sigma)
  \le \frac{\E_{P^\infty}\sigma}{\gamma}
\]
for every stopping time \(\sigma\), using one e-detector for the entire threshold family.  Fixed-scale e-detector universality \cite{ramdas-edetector-universality-2026} characterizes this stronger validity property.  Using an independent mean-one calibration clock, the common detector can be chosen, for any prescribed \(\eta>0\), so that every null-law average run length lies between \(\gamma\) and \((1+\eta)\gamma+1\).  Ordinary ARL alone is insufficient for either converse: we give a compact Polish example with \(\Gamma(Q)=0\), every null ARL in \([\gamma,\gamma+2)\), and robust Lorden delay at most two.

The examples isolate the roles of finite block length, countability, and pointwise rather than uniform separation.  The same simple scores also yield universal sequential goodness-of-fit, two-sample, and independence tests whenever the underlying sigma-algebras admit countable determining classes.  These applications are collected after the change-detection theorem.

\paragraph{Contributions.}
The main contributions are as follows.
\begin{itemize}[leftmargin=*,itemsep=2pt,topsep=3pt]
\item We give an assumption-free characterization, apart from the i.i.d. sampling model, of when a composite null is sequentially testable with power one against every member of a composite alternative.  The criterion is equivalent to countably many finite-block events, randomized tests, or bounded scores with pointwise positive alternative expectation.
\item We turn these finite-block certificates into both a single divergent e-process and a single test supermartingale on a deterministic reduced-time filtration, and conversely recover a countable family of certificates from any power-one stopping rule.  Equivalent conditions use bounded estimable law functionals of finite degree.  The resulting level-\(\alpha\) tests have finite pointwise expected sample size \(O_Q(\log(1/\alpha))\), while a geometric corollary gives a countable weak-$*$ closed-hull cover at positive total-variation distance.
\item We characterize simultaneous optimal growth, sharp expected stopping times, and all attainable lower rate profiles by countable recovery of finite-block log payoffs.  We give sufficient conditions for the intrinsic and raw KL benchmarks, and quantify the unavoidable deficits when simultaneous optimality fails.
\item We prove that the qualitative and quantitative conditions also characterize i.i.d. change detectability under fixed-time false-alarm control together with ordinary ARL control.  The constructions satisfy the stronger optional-horizon inequality: for every \(\eta>0\), one e-detector places every null ARL in \([\gamma,(1+\eta)\gamma+1]\) and attains the corresponding robust Lorden-delay bounds.  A calibrated counterexample shows why ordinary ARL alone does not suffice.
\item We characterize the uniform regime as well: one uniform finite-block gap is equivalent to uniform consistency and a uniform \(O(\log(1/\alpha))\) expected-sample-size bound.  We also show that two-observation certificates suffice whenever testability holds on a countable discrete sample space, as well as for every alternative outside a weakly compact null on a Polish space, and derive universal goodness-of-fit, two-sample, and independence-testing consequences.
\end{itemize}

\paragraph{Paper outline.}
\Cref{sec:setup} introduces the testing and e-process setup.  \Cref{sec:main} states and proves the finite-block characterization, gives its weak-$*$ geometric form, and develops examples clarifying finite-block and countability phenomena.  Sections~\ref{sec:quantitative-testing} and~\ref{sec:rate-profiles} give the quantitative characterization, sufficient conditions, and attainable rate profiles.  \Cref{sec:change-detection} introduces fixed-horizon and optional-horizon ARL control and proves the corresponding qualitative and quantitative detection results, including the distinction from ordinary ARL.  \Cref{sec:applications} records the goodness-of-fit, two-sample, and independence-testing applications, and \Cref{sec:conclusion} concludes with further directions.

\paragraph{Related work.}
Sequential tests of power one appear in classical work of  Robbins, Darling, Siegmund and Lai \cite{darling-robbins-1967,darling-robbins-1968,robbins-siegmund-1970,robbins-siegmund-1974,lai1977}.  General lower and upper bounds for their stopping times are developed by Agrawal, Ram and Ramdas \cite{agrawal-ramdas-2025}.  E-values, test martingales, and e-processes provide a modern language for optional stopping and anytime-valid inference \cite{shafer2011,howard2021,ramdas2023,grunwald2024}; composite e-processes and Ville-type characterizations are developed in \cite{ruf2023}.  Questions about when nontrivial p-values, e-values, or bounded e-variables exist for composite hypotheses are studied by Zhang, Ramdas and Wang and by Larsson, Ramdas and Ruf \cite{zhang-ramdas-wang-2024,larsson-ramdas-ruf-constraints-2026}.  The unrestricted fixed-sample closure theorem of Larsson, Ruf and Ramdas \cite{larsson-ruf-ramdas-2026} supplies the geometric identity used below, while recent GROW duality identifies the analogous weak-$*$ geometry for worst-case logarithmic growth \cite{ram-larsson-ruf-ramdas-grow-2026}.

Two recent papers are especially close.  Ram and Ramdas \cite{ram-ramdas-2026} prove that on a Polish sample space every weakly compact null class admits a power-one sequential test against its complement, construct a divergent e-process, and show that weak compactness is sufficient but not necessary.  Their later work \cite{ram-ramdas-optimal-growth-2026} gives, without compactness assumptions, a necessary-and-sufficient condition for power-one testing against a fixed point alternative, characterizes optimal betting growth, and emphasizes the sufficiency of test supermartingales on reduced filtrations.  The present paper identifies the additional countable-assembly conditions that are necessary and sufficient for one procedure to work, or to attain the optimal rates, over an arbitrary composite alternative; it also constructs a fine-time e-process, proves the fixed-horizon change-detection equivalences, and yields one common e-detector for the full threshold family.  Proposition~\ref{prop:weak-compact-degree-two} gives an independent short proof of the qualitative weak-compactness result and strengthens it by showing that bounded continuous two-observation certificates always suffice.  Proposition~\ref{prop:continuous-assembly} isolates the countable-assembly mechanism, while Example~\ref{ex:countable-cocountable} shows why regularity of the measurable space cannot simply be omitted.  Example~\ref{ex:lalley-nobel-rate-gap}, based on the indistinguishability theorem of Lalley and Nobel \cite{lalley-nobel-2003}, gives a failure of countable assembly even for Borel classes on a compact Polish space.  Remark~\ref{rem:janssen-local-power} distinguishes this fixed-alternative obstruction from Janssen's local-power restrictions \cite{janssen-2000}.

A growing constructive literature develops sequential nonparametric tests through betting.  Shekhar and Ramdas study nonparametric sequential two-sample testing using predictive approximations to variational witnesses \cite{shekhar-ramdas-2024}.  Kernel-based and prediction-based sequential independence and two-sample tests are developed by Podkopaev et al. and Podkopaev and Ramdas \cite{podkopaev-et-al-2023,podkopaev-ramdas-2023}; a rank-based sequential independence test is given by Henzi and Law \cite{henzi-law-2024}.  Broader constructions include deep anytime-valid tests for operator-defined null hypotheses \cite{pandeva-et-al-2024} and model-X sequential conditional-independence tests by betting \cite{shaer-et-al-2023}.  These works provide concrete algorithms and quantitative guarantees for important structured problems.  Our focus is complementary: we characterize, without domination, compactness, or parametric assumptions on \(\Pcal\) and \(\Qcal\), exactly when some power-one procedure exists and which pointwise growth profiles a single procedure can attain.

Quickest change detection has a long classical lineage.  Page's CUSUM, Shiryaev's Bayesian rule, and the Shiryaev--Roberts procedure are foundational constructions \cite{page1954,shiryaev1963,roberts1966}.  Lorden and Pollak introduced influential worst-case and conditional-delay criteria, and Moustakides proved exact CUSUM optimality under Lorden's formulation \cite{lorden1971,pollak1985,moustakides1986}.  Standard book-length treatments include Basseville and Nikiforov, Poor and Hadjiliadis, and Tartakovsky, Nikiforov and Basseville \cite{basseville-nikiforov-1993,poor-hadjiliadis-2008,tartakovsky-et-al-2014}; Veeravalli and Banerjee give a modern survey of Bayesian and minimax theory and its extensions, while Tartakovsky treats general non-i.i.d. models \cite{veeravalli-banerjee-2014,tartakovsky-2020}.  Most of this literature seeks optimal or asymptotically optimal procedures within specified stochastic models.  Shin, Ramdas and Rinaldo \cite{shin2024} introduced e-detectors, which convert restarted e-processes into nonparametric change detectors with nonasymptotic ARL guarantees.  Ramdas \cite{ramdas-edetector-universality-2026} proves that the original e-detectors are universal exactly for optional-horizon false-alarm control, while weak e-detectors are universal for bare ARL control.  This distinction lets us state the present characterization directly in terms of alarm-time validity; our construction strengthens the scale-by-scale representation theorem by producing one common e-detector.  Related assumption-light reductions from confidence sequences or sequential estimation to change detection appear in \cite{shekhar-ramdas-backward-2023,shekhar-ramdas-reduction-2024}, sharp composite lower bounds and matching bounded-mean results are developed in \cite{ram-ramdas-bounded-means-2026}, and non-partitioned e-detectors are studied in \cite{saha-ramdas-2026}.  The present paper identifies the exact qualitative and quantitative countable finite-block properties for power-one testing and change detection.  Fixed-time false-alarm bounds suffice for the converse implications, while the constructions retain optional-horizon validity; ordinary ARL alone gives a different problem.

\section{Setup}\label{sec:setup}

Let \((\X,\B)\) be a measurable space and let \(\M_1(\X)\) denote the set of probability measures on it.  Let \(\Pcal\subseteq\M_1(\X)\) be a null class and \(\Qcal\subseteq\M_1(\X)\) an alternative class.  For \(R\in\M_1(\X)\), write \(R^n\) and \(R^\infty\) for the i.i.d. product laws on \(\X^n\) and \(\X^{\mathbb N}\).  Let \((X_i)_{i\ge1}\) be the coordinate process, let \(\F_n=\sigma(X_1,\ldots,X_n)\) for \(n\ge1\), and let \(\F_0\) be the trivial sigma-algebra.

\begin{definition}[Sequential tests]
A sequential test is a stopping time \(\tau:\X^{\mathbb N}\to\{1,2,\ldots\}\cup\{\infty\}\) with respect to \((\F_n)\).  It has level at most \(\alpha\in(0,1)\) for \(\Pcal\) if
\[
  \sup_{P\in\Pcal}P^\infty(\tau<\infty)\le \alpha.
\]
It has power one against \(\Qcal\) if
\[
  Q^\infty(\tau<\infty)=1\qquad\text{for every }Q\in\Qcal.
\]
\end{definition}

\begin{definition}[E-processes]
An e-process for \(\Pcal\) is an adapted process \(\Eproc=(\Eproc_n)_{n\ge0}\) with values in \([0,\infty]\) such that, for every \(P\in\Pcal\) and every \((\F_n)_{n\ge0}\)-stopping time \(\tau\) taking values in \(\{0,1,\ldots\}\cup\{\infty\}\),
\[
  \E_{P^\infty}\Eproc_\tau\le 1,
\]
where \(\Eproc_\infty:=\liminf_{n\to\infty}\Eproc_n\).  We do not require \(\Eproc_0=1\); the displayed inequality implies \(\Eproc_0\le1\).  For every deterministic \(n\), an e-process is finite \(P^\infty\)-almost surely for each \(P\in\Pcal\).  If \(\limsup_n\Eproc_n=\infty\) under \(Q^\infty\), then \(\inf\{n:\Eproc_n\ge1/\alpha\}\) is a level-\(\alpha\), power-one test against \(Q\).
\end{definition}

\begin{lemma}[Bounded stopping times suffice]\label{lem:bounded-stopping-eprocess}
Let \(\Eproc=(\Eproc_n)_{n\ge0}\) be nonnegative and adapted, set \(\Eproc_\infty=\liminf_n\Eproc_n\), and fix \(P\in\Pcal\).  If \(\E_{P^\infty}\Eproc_\sigma\le1\) for every bounded stopping time \(\sigma\), then the same inequality holds for every stopping time \(\tau\).
\end{lemma}

\begin{proof}
The identity \(\liminf_{n\to\infty}\Eproc_{\tau\wedge n}=\Eproc_\tau\) holds pathwise: on \(\{\tau<\infty\}\) the sequence is eventually constant, and on \(\{\tau=\infty\}\) it is the defining liminf.  Fatou's lemma therefore gives
\[
  \E_{P^\infty}\Eproc_\tau
  \le \liminf_{n\to\infty}\E_{P^\infty}\Eproc_{\tau\wedge n}
  \le1.
\]
\end{proof}

For an event \(B\in\B^{\otimes n}\), define its worst-case null probability
\[
  \pi_n(B):=\sup_{P\in\Pcal}P^n(B).
\]
For a randomized fixed-sample test \(\varphi:\X^n\to[0,1]\), define
\[
  \size(\varphi):=\sup_{P\in\Pcal}\E_{P^n}\varphi,
  \qquad
  \pow_Q(\varphi):=\E_{Q^n}\varphi.
\]
For a deterministic test \(\psi=\one_B\), these are \(\pi_n(B)\) and \(Q^n(B)\).

A law functional \(F:\M_1(\X)\to\R\) is called \emph{bounded estimable of degree at most \(n\)} if
\[
  F(R)=\E_{R^n}h
\]
for some bounded measurable symmetric kernel \(h:\X^n\to\R\).  This is the classical notion of an estimable functional underlying the theory of U-statistics \cite{halmos1946,hoeffding1948}.

\section{Main theorem}\label{sec:main}

\begin{theorem}[Finite-block characterizations of power-one testing]\label{thm:main}
The following statements are equivalent.
\begin{enumerate}[label=(\roman*)]
\item For some \(\alpha\in(0,1)\), there exists a level-\(\alpha\) sequential test with power one against every \(Q\in\Qcal\).

\item For every \(\alpha\in(0,1)\), there exists a level-\(\alpha\) sequential test with power one against every \(Q\in\Qcal\).

\item There exists an e-process \(\Eproc=(\Eproc_n)_{n\ge0}\) for \(\Pcal\) such that
\[
  \Eproc_n\to\infty
  \qquad Q^\infty\text{-almost surely for every }Q\in\Qcal.
\]

\item[(iii$'$)] There exists an e-process \(\Eproc=(\Eproc_n)_{n\ge0}\) for \(\Pcal\) such that
\[
  \limsup_{n\to\infty}\Eproc_n=\infty
  \qquad Q^\infty\text{-almost surely for every }Q\in\Qcal.
\]

\item There exists a countable family \((n_m,B_m)_{m\ge1}\), with \(n_m\ge1\) and \(B_m\in\B^{\otimes n_m}\), such that for every \(Q\in\Qcal\) there is an \(m\) satisfying
\[
  Q^{n_m}(B_m)>\pi_{n_m}(B_m).
\]

\item There exists a countable family \((n_m,\varphi_m)_{m\ge1}\) of randomized fixed-sample tests, \(\varphi_m:\X^{n_m}\to[0,1]\), such that for every \(Q\in\Qcal\) there is an \(m\) satisfying
\[
  \pow_Q(\varphi_m)>\size(\varphi_m).
\]
Equivalently, after enumerating rational declared levels, there is a countable family \((n_m,\varphi_m,q_m)\) with \(q_m\in\mathbb Q\cap[0,1]\) such that for every \(Q\in\Qcal\) there is an \(m\) satisfying
\[
  \size(\varphi_m)\le q_m<\pow_Q(\varphi_m).
\]

\item There exists a countable family \((n_m,h_m)_{m\ge1}\) of bounded measurable scores \(h_m:\X^{n_m}\to\R\) such that
\[
  \sup_{P\in\Pcal}\E_{P^{n_m}}h_m\le0
  \quad\text{for every }m,
\]
and for every \(Q\in\Qcal\) there is an \(m\) satisfying
\[
  \E_{Q^{n_m}}h_m>0.
\]

\item There exists a countable family \((F_m)_{m\ge1}\) of bounded estimable law functionals of finite degrees such that
\[
  F_m(P)\le0
  \quad\text{for every }P\in\Pcal\text{ and every }m,
\]
and every \(Q\in\Qcal\) satisfies \(F_m(Q)>0\) for at least one \(m\).

\item There exist deterministic nondecreasing integers
\[
  0=t_0\le t_1\le t_2\le\cdots,
  \qquad t_k\le k,
  \qquad \frac{t_k}{k}\longrightarrow1,
\]
and a nonnegative process \(\Sproc=(\Sproc_k)_{k\ge0}\), adapted to the full-data filtration \((\F_{t_k})_{k\ge0}\), such that \(\Sproc_0=1\), and for every \(P\in\Pcal\) the process is integrable and a supermartingale under \(P^\infty\), and
\[
  \Sproc_k\longrightarrow\infty
  \qquad Q^\infty\text{-almost surely for every }Q\in\Qcal.
\]
\end{enumerate}
\end{theorem}

\paragraph{Proof roadmap.}
The forward direction extracts the countable rejection events \(\{\tau\le n\}\) from any power-one stopping rule and then centers their randomized-test versions to obtain bounded finite-block scores.  The reverse direction turns each score into positive blockwise betting factors, verifies optional-stopping validity on every block offset, averages those offsets to obtain a fine-time e-process, and mixes the countable family.  Conditions~\textup{(vii)} and \textup{(viii)} record, respectively, the estimable-functional and reduced-time-supermartingale forms of the same construction.

\begin{proof}
We prove
\[
  (iii)\Rightarrow(iii')\Rightarrow(ii)\Rightarrow(i)\Rightarrow(iv)
  \Rightarrow(v)\Rightarrow(vi)\Rightarrow(iii),
\]
and then \((vi)\Longleftrightarrow(vii)\) and \((vi)\Rightarrow(viii)\Rightarrow(ii)\).

The implication \((iii)\Rightarrow(iii')\) is immediate.  Assume \((iii')\) and define
\[
  \tau_\alpha:=\inf\{n:\Eproc_n\ge1/\alpha\}.
\]
Since \(\Eproc_0\le1<1/\alpha\), one has \(\tau_\alpha\ge1\).  The e-process property gives
\[
  \frac1\alpha P^\infty(\tau_\alpha<\infty)
  \le \E_{P^\infty}\Eproc_{\tau_\alpha}\le1
\]
for every \(P\in\Pcal\).  The limsup condition forces \(\tau_\alpha<\infty\) \(Q^\infty\)-almost surely for every \(Q\in\Qcal\).  This proves (ii), and \((ii)\Rightarrow(i)\) is immediate.

Assume (i), and let \(\tau\) be a level-\(\alpha\), power-one sequential test for some \(\alpha<1\).  For every \(n\ge1\), the event \(\{\tau\le n\}\in\F_n\) can be written as \(\{(X_1,\ldots,X_n)\in B_n\}\) for some \(B_n\in\B^{\otimes n}\).  For every \(P\in\Pcal\),
\[
  P^n(B_n)=P^\infty(\tau\le n)
  \le P^\infty(\tau<\infty)\le\alpha,
\]
so \(\pi_n(B_n)\le\alpha\).  For every \(Q\in\Qcal\),
\[
  Q^n(B_n)=Q^\infty(\tau\le n)
  \uparrow Q^\infty(\tau<\infty)=1.
\]
Thus, for each fixed \(Q\), some \(n\) satisfies \(Q^n(B_n)>\alpha\ge\pi_n(B_n)\), proving (iv).

The implication \((iv)\Rightarrow(v)\) follows by taking \(\varphi_m=\one_{B_m}\).  For \((v)\Rightarrow(vi)\), set
\[
  h_m:=\varphi_m-\size(\varphi_m).
\]
Then \(h_m\) is bounded, \(\sup_{P\in\Pcal}\E_{P^{n_m}}h_m\le0\), and the strict power--size gap is exactly \(\E_{Q^{n_m}}h_m>0\).  For the rational-level formulation, replace each \(\varphi_m\) by copies indexed by all \(q\in\mathbb Q\cap[\size(\varphi_m),1]\), and enumerate the resulting countable family.  Whenever \(\pow_Q(\varphi_m)>\size(\varphi_m)\), density of the rationals supplies one such \(q\) with \(\size(\varphi_m)\le q<\pow_Q(\varphi_m)\).  The converse is immediate.

We prove \((vi)\Rightarrow(iii)\).  Write \(N_m=n_m\), let
\(a_m:=\inf h_m\), and put \(H_m:=\lVert h_m\rVert_\infty\).  Define
\[
  \mathcal K_m:=
  \begin{cases}
    \mathbb Q\cap(0,-1/a_m),&a_m<0,\\
    \mathbb Q\cap(0,\infty),&a_m\ge0.
  \end{cases}
\]
For \(\lambda\in\mathcal K_m\), the block factor
\[
  G_{m,\lambda}:=1+\lambda h_m
\]
is strictly positive and bounded, and
\[
  \E_{P^{N_m}}G_{m,\lambda}
  =1+\lambda\E_{P^{N_m}}h_m\le1
  \qquad(P\in\Pcal).
\]

Fix \(s\in\{0,1,\ldots,N_m-1\}\).  Put \(\Pi^{m,\lambda,s}_0=1\) and, for \(j\ge1\),
\[
  \Pi^{m,\lambda,s}_j
  :=\prod_{r=1}^{j}
  G_{m,\lambda}\bigl(
    X_{s+(r-1)N_m+1},\ldots,X_{s+rN_m}
  \bigr).
\]
Under every \(P\in\Pcal\), this is a nonnegative supermartingale for the grid filtration \(\mathcal G^{m,s}_j:=\F_{s+jN_m}\).  We now embed this grid process in ordinary time by retaining its value on the grid and setting it to zero elsewhere; the stopping-time validity of this device is checked directly below.  Define
\[
  W^{m,\lambda,s}_t:=
  \begin{cases}
    \Pi^{m,\lambda,s}_j,&t=s+jN_m\text{ for some }j\ge0,\\
    0,&\text{otherwise.}
  \end{cases}
\]
For a bounded stopping time \(\sigma\), let \(J=j\) on \(\{\sigma=s+jN_m\}\) and \(J=\infty\) if \(\sigma\) does not lie on the grid.  This is a stopping time for \((\mathcal G_j^{m,s})\), since for every \(j\),
\[
  \{J\le j\}=\bigcup_{r=0}^j\{\sigma=s+rN_m\}
  \in\F_{s+jN_m}=\mathcal G_j^{m,s}.
\]
For every integer \(K\), optional sampling and nonnegativity give
\[
  \E_{P^\infty}
  \bigl[\Pi^{m,\lambda,s}_J\one\{J\le K\}\bigr]
  \le \E_{P^\infty}\Pi^{m,\lambda,s}_{J\wedge K}\le1.
\]
Letting \(K\to\infty\) and using monotone convergence yields
\[
  \E_{P^\infty}W^{m,\lambda,s}_\sigma
  =\E_{P^\infty}
  \bigl[\Pi^{m,\lambda,s}_J\one\{J<\infty\}\bigr]
  \le1.
\]
Lemma~\ref{lem:bounded-stopping-eprocess} shows that each \(W^{m,\lambda,s}\) is an e-process.  Moreover, for every bounded stopping time \(\sigma\), linearity gives
\[
  \E_{P^\infty}V^{m,\lambda}_\sigma\le1,
  \qquad
  V^{m,\lambda}_t
  :=\frac1{N_m}\sum_{s=0}^{N_m-1}W^{m,\lambda,s}_t.
\]
A further application of Lemma~\ref{lem:bounded-stopping-eprocess} shows that \(V^{m,\lambda}\) is an e-process.  This bounded-stopping-time argument is needed because the liminf convention at time infinity does not commute with averaging.

Choose positive weights \(w_{m,\lambda}\) whose sum is at most one and define the possibly extended-valued process
\[
  \Eproc_t:=\sum_m\sum_{\lambda\in\mathcal K_m}
  w_{m,\lambda}V^{m,\lambda}_t.
\]
For every bounded stopping time \(\sigma\), Tonelli's theorem gives
\[
  \E_{P^\infty}\Eproc_\sigma
  \le\sum_m\sum_{\lambda\in\mathcal K_m}w_{m,\lambda}\le1.
\]
Lemma~\ref{lem:bounded-stopping-eprocess} proves that \(\Eproc\) is an e-process.  In particular, it is finite at every deterministic time almost surely under each null; allowing \(+\infty\) off the null is harmless.

\emph{Alternative-wise divergence.}
Fix \(Q\in\Qcal\), and choose \(m\) with
\(\mu_m:=\E_{Q^{N_m}}h_m>0\).  Because \(\mu_m>0\), the score \(h_m\) is not identically zero and hence \(H_m>0\).  There is therefore a rational \(\lambda\in\mathcal K_m\) such that
\[
  \lambda H_m\le\frac12,
  \qquad
  \lambda H_m^2\le\frac{\mu_m}{2}.
\]
The elementary inequality \(\log(1+u)\ge u-u^2\), valid for \(|u|\le1/2\), then gives
\[
  g_{m,\lambda}:=\E_{Q^{N_m}}\log G_{m,\lambda}
  \ge \lambda\mu_m-\lambda^2\E_{Q^{N_m}}h_m^2
  \ge\frac{\lambda\mu_m}{2}>0.
\]
For every offset \(s\), the strong law for independent blocks gives
\[
  \frac1j\log\Pi^{m,\lambda,s}_j
  \longrightarrow g_{m,\lambda}>0
  \qquad Q^\infty\text{-almost surely.}
\]
The convergence holds simultaneously for the finitely many offsets.  Hence, for every \(M>0\), there is an almost surely finite \(J_M\) such that \(\Pi_j^{m,\lambda,s}\ge N_mM\) for every offset \(s\) and every \(j\ge J_M\).  At every ordinary time \(t\), exactly one offset contributes, namely
\[
  V^{m,\lambda}_t
  =\frac1{N_m}
  \Pi^{m,\lambda,\,t\bmod N_m}_{\lfloor t/N_m\rfloor}.
\]
Since \(\lfloor t/N_m\rfloor\to\infty\), the preceding simultaneous bound shows that \(V^{m,\lambda}_t\to\infty\), and therefore
\(\Eproc_t\to\infty\), \(Q^\infty\)-almost surely.

The equivalence \((vi)\Longleftrightarrow(vii)\) follows directly from the definition of an estimable functional.  Indeed, for a bounded kernel \(h:\X^n\to\R\), its symmetrization
\[
  h_{\mathrm{sym}}(x_1,\ldots,x_n)
  :=\frac1{n!}\sum_{\pi}h(x_{\pi(1)},\ldots,x_{\pi(n)})
\]
has the same expectation as \(h\) under every product law \(R^n\), and is bounded, measurable, and symmetric.

We next prove \((vi)\Rightarrow(viii)\).  Enumerate a countable collection of the factors just constructed as \((G_i,N_i)_{i\ge1}\), retaining enough factors that every \(Q\in\Qcal\) has positive expected log factor for at least one \(i\).  Choose \(v_i>0\) with \(\sum_i v_i=1\), put \(d_i=2^iN_i\), and define, for \(k\ge0\),
\[
  t_k:=\sum_{i\ge1}\left\lfloor\frac{k}{d_i}\right\rfloor N_i.
\]
At macro-stage \(k\ge1\), let \(I_k=\{i:d_i\mid k\}\).  This set is finite because \(d_i\ge2^i\).  Update every component in \(I_k\), and leave the others unchanged.  More precisely, order \(I_k\) increasingly and partition the consecutive fresh observations
\[
  X_{t_{k-1}+1},\ldots,X_{t_k}
\]
into consecutive blocks of lengths \((N_i)_{i\in I_k}\) in that order, assigning the length-\(N_i\) block to component \(i\).  This is possible because
\[
  t_k-t_{k-1}=\sum_{i\in I_k}N_i.
\]
Thus all blocks used at a stage are mutually disjoint, all observations exposed through time \(t_k\) have been allocated, and the new blocks are independent of \(\F_{t_{k-1}}\).  Moreover,
\[
  t_k\le k\sum_{i\ge1}2^{-i}=k.
\]
For every fixed \(I\),
\[
  \frac{t_k}{k}
  \ge \sum_{i=1}^I2^{-i}-\frac1k\sum_{i=1}^IN_i,
\]
so \(\liminf_k t_k/k\ge1-2^{-I}\).  Letting \(I\to\infty\) proves \(t_k/k\to1\).

Let \(Z^i_k\) be the product of the factors assigned to component \(i\) through stage \(k\), with \(Z^i_0=1\), and set
\[
  \Sproc_k:=\sum_{i\ge1}v_iZ^i_k.
\]
The blocks used at stage \(k\) are independent of \(\F_{t_{k-1}}\), hence
\[
  \E_{P^\infty}[Z^i_k\mid\F_{t_{k-1}}]
  \le Z^i_{k-1}
\]
for every \(P\in\Pcal\).  Conditional monotone convergence yields
\[
  \E_{P^\infty}[\Sproc_k\mid\F_{t_{k-1}}]
  \le\Sproc_{k-1}.
\]
Thus \(\Sproc\) is a test supermartingale for the full-data filtration sampled at \((t_k)\), simultaneously under every null law.

Fix \(Q\in\Qcal\), and choose \(i\) with
\(\E_{Q^{N_i}}\log G_i>0\).  Component \(i\) is updated \(\lfloor k/d_i\rfloor\) times by stage \(k\) on independent \(Q^{N_i}\)-blocks, so the strong law gives \(Z^i_k\to\infty\), and hence \(\Sproc_k\to\infty\), \(Q^\infty\)-almost surely.

Finally assume (viii).  Define
\[
  \kappa_\alpha:=\inf\{k:\Sproc_k\ge1/\alpha\},
  \qquad
  \tau_\alpha:=t_{\kappa_\alpha},
\]
with the convention \(t_\infty=\infty\).
If \(t_k=0\), then \(\Sproc_k\) is \(\F_0\)-measurable and hence constant.  Since \(\Sproc\) is a supermartingale under every null law,
\[
  \Sproc_k=\E_{P^\infty}\Sproc_k
  \le \E_{P^\infty}\Sproc_0=1<1/\alpha
\]
for any \(P\in\Pcal\).  Thus \(\kappa_\alpha\) cannot occur at a zero grid time and \(\tau_\alpha\ge1\).  Since \(t_k\) is deterministic and tends to infinity, with
\(K(n):=\max\{k:t_k\le n\}\) one has
\[
  \{\tau_\alpha\le n\}
  =\{\kappa_\alpha\le K(n)\}\in\F_n.
\]
Thus \(\tau_\alpha\) is a stopping time in the original filtration.  Ville's inequality gives
\[
  \sup_{P\in\Pcal}P^\infty(\tau_\alpha<\infty)\le\alpha,
\]
and divergence of \(\Sproc_k\) gives power one under every alternative.
\end{proof}

\begin{corollary}[Logarithmic expected sample size]\label{cor:testing-ess}
Whenever the equivalent conditions of Theorem~\ref{thm:main} hold, the level-\(\alpha\) tests may be chosen so that, for every fixed \(Q\in\Qcal\), there are finite constants \(A_Q,B_Q\) satisfying
\[
  \E_{Q^\infty}\tau_\alpha
  \le A_Q+B_Q\log(1/\alpha),
  \qquad 0<\alpha<1.
\]
In particular, power one can always be achieved with finite pointwise expected sample size.
\end{corollary}

\begin{proof}
Use the e-process from the proof of Theorem~\ref{thm:main}.  Fix \(Q\), and choose a component \((m,\lambda)\) with block length \(N=N_m\), mixture weight \(w=w_{m,\lambda}>0\), and positive log drift
\[
  g=\E_{Q^N}\log G_{m,\lambda}>0.
\]
On the offset-zero grid, put
\[
  Y_j:=\log\Pi^{m,\lambda,0}_j.
\]
At time \(jN\),
\[
  \Eproc_{jN}\ge \frac{w}{N}e^{Y_j}.
\]
Thus, with \(b_\alpha=\log(N/(\alpha w))>0\) and
\(\sigma_b=\inf\{j\ge1:Y_j\ge b\}\), one has
\(\tau_\alpha\le N\sigma_{b_\alpha}\).  The increments of \(Y_j\) are bounded, i.i.d., and have mean \(g\).  If they take values in an interval of length \(\Delta>0\), then for \(j\ge2b/g\), the event \(\{\sigma_b>j\}\) implies \(Y_j-jg\le-jg/2\), and Hoeffding's inequality \cite{hoeffding1963} gives
\[
  Q^\infty(\sigma_b>j)
  \le \exp\left(-\frac{jg^2}{2\Delta^2}\right).
\]
Writing \(j_0=\lceil2b/g\rceil\) and \(c=g^2/(2\Delta^2)\), the tail-sum formula yields
\[
  \E_{Q^\infty}\sigma_b
  =\sum_{j\ge0}Q^\infty(\sigma_b>j)
  \le j_0+\sum_{j\ge j_0}e^{-cj}
  \le \frac{2b}{g}+1+\frac1{1-e^{-c}}.
\]
If \(\Delta=0\), every increment equals \(g\) and \(\sigma_b=\lceil b/g\rceil\).  Substituting \(b=b_\alpha\) proves the claim.
\end{proof}

\begin{remark}[A testing information lower bound]\label{rem:information-lower-bounds}
Classical and recent expected-sample-size analyses include \cite{robbins-siegmund-1974,agrawal-ramdas-2025}.  In a general composite problem, the sharp pointwise benchmark is the intrinsic block-growth rate \(\Gamma(Q)\), defined in \eqref{eq:intrinsic-growth-rate}, which can be smaller than \(\inf_{P\in\Pcal}\KL(Q\|P)\).  Proposition~\ref{prop:sharp-power-one} proves the assumption-free bound
\[
  \E_{Q^\infty}\tau_\alpha\ge\frac{\log(1/\alpha)}{\Gamma(Q)}
  \qquad(0<\Gamma(Q)<\infty),
\]
and characterizes simultaneous attainment of its leading constant.  Its proof repeats the test on fresh data and extracts finite-block events, so it needs neither domination nor a stopped likelihood identity.
\end{remark}

\begin{remark}[E-process versus supermartingale.]\label{rem:offsets}
Condition~\textup{(viii)} shows that e-processes are not needed merely for existence in the i.i.d. model: one test supermartingale suffices on the deterministic reduced-time filtration \((\F_{t_k})\), and at macro-stage \(k\) it has used all observations through time \(t_k\).  The relation \(t_k/k\to1\) says that the observation budget is asymptotically one observation per macro-stage.  It does \emph{not} assert that the set of distinct inspection times \(\{t_k:k\ge0\}\) has natural density one: the nondecreasing sequence may repeat values and skip ordinary times.  The fine-time e-process remains useful because it is valid at every stopping time in the original filtration and the offset average gives full-time convergence \(\Eproc_n\to\infty\), rather than only the limsup property in condition~\textup{(iii$'$)}.

It is natural to ask whether one can always take \(t_k=k\), that is, whether every sequentially testable i.i.d. problem admits a nonnegative \((\F_n)\)-supermartingale diverging under every alternative.  The obstruction is that a block factor need not have conditional mean at most one when only part of its block has been revealed.  In Example~\ref{ex:two-observations} the answer is nevertheless positive: for any \(\lambda>0\),
\[
  \prod_{i=2}^n\bigl(1+\lambda\one\{X_i\ne X_{i-1}\}\bigr)
\]
is identically one under every point-mass null and has positive almost-sure exponential growth under every non-point-mass alternative.  To see the latter without invoking a dependent strong law, split the indicators \(\one\{X_i\ne X_{i-1}\}\) into the even and odd subsequences.  Within each subsequence the indicators are i.i.d., because they use disjoint pairs, and both have mean \(Q^2(X_1\ne X_2)>0\); applying the strong law to the two subsequences proves the claimed growth.
\end{remark}

\begin{remark}[Estimable functionals and block degree]\label{rem:polynomial}
Condition~\textup{(vii)} places the theorem in the classical theory of estimable functionals and U-statistics \cite{halmos1946,hoeffding1948}.  After symmetrization, every bounded block score \(h\) of length \(n\) defines the bounded estimable functional
\[
  F(Q)=\E_{Q^n}h,
\]
and the associated U-statistic is an unbiased estimator of \(F(Q)\).  In algebraic language, \(F\) is a bounded homogeneous polynomial law functional of degree at most \(n\).  The hierarchy is nested, since a degree-\(n\) certificate can be viewed as degree \(n+1\) by ignoring the last observation.  Thus one may define the minimal degree \(n^*(Q;\Pcal)\) of an individual certificate.

Example~\ref{ex:two-observations} has degree two but not degree one, whereas the mean example in Section~\ref{sec:change-detection} has degree one and the independence example in Section~\ref{sec:applications} has degree two.  Propositions~\ref{prop:countable-space-degree-two} and~\ref{prop:weak-compact-degree-two} below show that degree two suffices whenever testability holds on a countable discrete sample space and for every alternative outside a weakly compact null on a Polish space.  This sharpens the remaining structural question: does degree two always suffice whenever sequential testability holds, or can \(\sup_{Q\in\Qcal}n^*(Q;\Pcal)>2\) occur?  Any counterexample must have a composite null and cannot live on a countable discrete sample space; within the Polish setting the null must also fail weak compactness.
\end{remark}

\begin{corollary}[Complete characterization for a simple null]\label{cor:simple-null-one-block}
Suppose \(\Pcal=\{P_0\}\).  The following are equivalent:
\begin{enumerate}[label=(\roman*)]
\item the equivalent conditions of Theorem~\ref{thm:main} hold;
\item there exists a countable class \(\mathscr A\subseteq\B\) such that every \(Q\in\Qcal\) satisfies \(Q(A)\ne P_0(A)\) for some \(A\in\mathscr A\);
\item condition~\textup{(vi)} of Theorem~\ref{thm:main} can be witnessed by a countable family of one-observation scores of the form
\[
  \one_A-P_0(A)
  \quad\text{or}\quad
  P_0(A)-\one_A.
\]
\end{enumerate}
Thus a simple null never requires genuinely longer blocks.
\end{corollary}

\begin{proof}
The implication (iii)\(\Rightarrow\)(i) is Theorem~\ref{thm:main}, and (ii)\(\Rightarrow\)(iii) follows by including both signed scores for every \(A\in\mathscr A\).

For (i)\(\Rightarrow\)(ii), let \((n_m,h_m)\) witness condition~\textup{(vi)}.  Every bounded \(\B^{\otimes n_m}\)-measurable function is measurable with respect to the product sigma-algebra generated by some countable algebra \(\mathscr A_m\subseteq\B\).  Indeed, every set in a sigma-algebra generated by an arbitrary family belongs to the sigma-algebra generated by a countable subfamily: the sets having this property themselves form a sigma-algebra.  Apply this observation to each rational superlevel set of \(h_m\), collect the coordinate sets appearing in the resulting countably many rectangles, and close them under finite Boolean operations.  If \(Q\) agrees with \(P_0\) on \(\mathscr A_m\), then the measures agree on \(\sigma(\mathscr A_m)\), and their \(n_m\)-fold products agree on the sigma-algebra with respect to which \(h_m\) is measurable.  Hence \(\E_{Q^{n_m}}h_m=\E_{P_0^{n_m}}h_m\).  Every \(Q\) separated by \(h_m\) must therefore differ from \(P_0\) on some \(A\in\mathscr A_m\).  The countable union \(\mathscr A=\bigcup_m\mathscr A_m\) proves (ii).
\end{proof}

\begin{remark}[The degree-one boundary]\label{rem:degree-one}
For an individual alternative \(Q\), a one-observation certificate exists exactly when
\[
  Q\notin\overline{\operatorname{co}(\Pcal)}^{\,w^*},
\]
in the dual of the bounded measurable functions; equivalently, \(Q\notin\mathscr C_1\) in the notation of \eqref{eq:block-hulls} below.  This is the fixed-sample existence question studied in \cite{zhang-ramdas-wang-2024,larsson-ruf-ramdas-2026}.  Theorem~\ref{thm:main} says that sequential testability is obtained by allowing the union over all finite degrees and requiring a countable assembly of the resulting positivity sets.
\end{remark}

\subsection{Fixed-block geometry and the countable cover}
\label{subsec:geometric-corollary}

Fix \(n\ge1\), and let \(\mathcal L_n\) be the Banach space of bounded \(\B^{\otimes n}\)-measurable functions on \(\X^n\), equipped with the supremum norm, and let \(\mathrm{ba}_n=\mathcal L_n^*\), the bounded finitely additive signed measures, or charges, on \((\X^n,\B^{\otimes n})\).  Equip \(\mathrm{ba}_n\) with the weak-$*$ topology \(\sigma(\mathrm{ba}_n,\mathcal L_n)\).  For a nonempty subfamily \(\mathcal A\subseteq\Qcal\), set
\begin{equation}\label{eq:block-hulls}
  \mathscr C_n:=\overline{\operatorname{co}\{P^n:P\in\Pcal\}}^{\,w^*},
  \qquad
  \mathscr D_n(\mathcal A):=\overline{\operatorname{co}\{Q^n:Q\in\mathcal A\}}^{\,w^*}.
\end{equation}
Under the standard identification \(\mathcal L_n^*=\mathrm{ba}_n\) \cite{rao-rao1983}, the positive mass-one charges form a weak-$*$ closed subset of the dual unit ball.  Banach--Alaoglu therefore makes this set weak-$*$ compact, so every element of \(\mathscr C_n\) and \(\mathscr D_n(\mathcal A)\) is a finitely additive probability measure.  For such charges, write
\[
  d_{\TV}(\mu,\nu)
  :=\sup_{0\le f\le1}\bigl(\E_\mu f-\E_\nu f\bigr),
\]
and use the infimum over pairs for the distance between sets.

\begin{proposition}[Fixed-block minimax identity]\label{prop:block-tv-identity}
For every nonempty \(\mathcal A\subseteq\Qcal\),
\begin{equation}\label{eq:block-tv-identity}
  \sup_{0\le\varphi\le1}
  \left\{
    \inf_{Q\in\mathcal A}\E_{Q^n}\varphi
    -\sup_{P\in\Pcal}\E_{P^n}\varphi
  \right\}
  =d_{\TV}\bigl(\mathscr C_n,\mathscr D_n(\mathcal A)\bigr).
\end{equation}
The infimum defining the distance on the right is attained; attainment of the supremum over \(\varphi\) is not claimed.
\end{proposition}

\begin{proof}
The sets \(\mathscr C_n\) and \(\mathscr D_n(\mathcal A)\) are weak-$*$ compact and convex: they are weak-$*$ closed subsets of the compact positive mass-one set described above.  Apply Fan's minimax theorem \cite{fan1953} to
\[
  F((\mu,\nu),\varphi)=\E_\nu\varphi-\E_\mu\varphi
\]
on the compact convex set \(\mathscr C_n\times \mathscr D_n(\mathcal A)\) and the convex set of tests \(0\le\varphi\le1\).  For fixed \(\varphi\), the payoff is weak-$*$ continuous and affine in \((\mu,\nu)\); for fixed \((\mu,\nu)\), it is affine, hence both quasiconcave and quasiconvex, in \(\varphi\).  Thus Fan's hypotheses hold and
\[
  \inf_{\mu\in \mathscr C_n,\,\nu\in \mathscr D_n(\mathcal A)}
  \sup_{0\le\varphi\le1}F((\mu,\nu),\varphi)
  =
  \sup_{0\le\varphi\le1}
  \inf_{\mu\in \mathscr C_n,\,\nu\in \mathscr D_n(\mathcal A)}F((\mu,\nu),\varphi).
\]
Because \(\mu\) and \(\nu\) have equal mass one, the inner one-sided supremum is the usual total-variation distance: replacing \(\varphi\) by \(1-\varphi\) reverses its sign.  On the right, affine weak-$*$ continuity shows that convexification and weak-$*$ closure do not change the relevant extrema, yielding the left side of \eqref{eq:block-tv-identity}.  Finally, \((\mu,\nu)\mapsto d_{\TV}(\mu,\nu)\) is the supremum of weak-$*$ continuous functions and therefore lower semicontinuous; compactness gives attainment.  This is the finite-block specialization of \cite{larsson-ruf-ramdas-2026}.
\end{proof}

\begin{corollary}[Geometric characterization of power-one testing]
\label{cor:geometric-power-one}
The equivalent conditions of Theorem~\ref{thm:main} hold if and only if there exist nonempty subfamilies \(\mathcal A_m\subseteq\Qcal\) and integers \(n_m\ge1\) such that
\begin{equation}\label{eq:countable-tv-cover}
  \Qcal=\bigcup_{m\ge1}\mathcal A_m,
  \qquad
  d_{\TV}\bigl(\mathscr C_{n_m},\mathscr D_{n_m}(\mathcal A_m)\bigr)>0
  \quad\text{for every }m.
\end{equation}
\end{corollary}

\begin{proof}
Assume condition~\textup{(v)} of Theorem~\ref{thm:main}, witnessed by \((n_m,\varphi_m)\).  For integers \(r\ge1\), define
\[
  \mathcal A_{m,r}:=\left\{Q\in\Qcal:
  \E_{Q^{n_m}}\varphi_m-
  \sup_{P\in\Pcal}\E_{P^{n_m}}\varphi_m\ge\frac1r\right\}.
\]
Every positive gap is at least \(1/r\) for some \(r\), so the nonempty \(\mathcal A_{m,r}\)'s form a countable cover.  For every nonempty class \(\mathcal A_{m,r}\), Proposition~\ref{prop:block-tv-identity} gives distance at least \(1/r\).

Conversely, suppose \eqref{eq:countable-tv-cover} holds and set
\(d_m=d_{\TV}(\mathscr C_{n_m},\mathscr D_{n_m}(\mathcal A_m))>0\).  By \eqref{eq:block-tv-identity}, choose a randomized test whose uniform gap exceeds \(d_m/2\).  These countably many tests verify condition~\textup{(v)} of Theorem~\ref{thm:main}.
\end{proof}

\subsection{Uniform versus pointwise separation}

\begin{proposition}[Uniform sequential testability]\label{prop:uniform-testing}
The following statements are equivalent.
\begin{enumerate}[label=(\roman*)]
\item For some \(n\ge1\), there is a randomized test \(\varphi:\X^n\to[0,1]\) such that
\[
  \inf_{Q\in\Qcal}\E_{Q^n}\varphi
  >\sup_{P\in\Pcal}\E_{P^n}\varphi.
\]
\item For some \(n\ge1\), \(d_{\TV}(\mathscr C_n,\mathscr D_n(\Qcal))>0\).
\item There are finite constants \(A,B\), independent of \(Q\) and \(\alpha\), such that for every \(\alpha\in(0,1)\) there is a level-\(\alpha\), power-one test \(\tau_\alpha\) satisfying
\[
  \sup_{Q\in\Qcal}\E_{Q^\infty}\tau_\alpha
  \le A+B\log(1/\alpha).
\]
\item For some \(\alpha\in(0,1)\), there is a level-\(\alpha\) test \(\tau\) such that
\[
  \sup_{Q\in\Qcal}Q^\infty(\tau>k)\longrightarrow0.
\]
\end{enumerate}
Whenever these conditions hold, the tests in (iii) may be chosen with a tail that, after a constant multiple of \(\log(1/\alpha)\), decays exponentially in the number of observed blocks, uniformly over \(Q\in\Qcal\).
\end{proposition}

\begin{proof}
The equivalence (i)\(\Longleftrightarrow\)(ii) is Proposition~\ref{prop:block-tv-identity}.  Assume (i), put
\[
  c=\sup_{P\in\Pcal}\E_{P^n}\varphi,
  \qquad
  \delta=\inf_{Q\in\Qcal}\E_{Q^n}\varphi-c>0,
\]
and let \(h=\varphi-c\), so \(|h|\le1\).  Since \(0<\delta\le1\), the choice \(\lambda=\delta/2\) satisfies \(\lambda\le1/2\).  The factor \(G=1+\lambda h\) is therefore positive, has null expectation at most one, and obeys
\[
  g_0:=\inf_{Q\in\Qcal}\E_{Q^n}\log G
  \ge\lambda\delta-\lambda^2=\frac{\delta^2}{4}.
\]
Let \(L_j\) be the product of \(G\) over the first \(j\) consecutive, independent \(n\)-blocks, set \(b_\alpha=\log(1/\alpha)\), and define \(\sigma_\alpha=\inf\{j:\log L_j\ge b_\alpha\}\) and \(\tau_\alpha=n\sigma_\alpha\).  Under every null, \((L_j)\) is a nonnegative supermartingale, so Ville's inequality makes \(\tau_\alpha\) level \(\alpha\).  Let \(\Delta_G<\infty\) be the range length of \(\log G\), which is independent of \(Q\).  In fact \(\Delta_G>0\): otherwise \(G\) would be constant, and its null expectation bound would force \(\log G\le0\), contradicting \(g_0>0\).  For every \(Q\in\Qcal\) and \(j\ge2b_\alpha/g_0\), Hoeffding's inequality \cite{hoeffding1963} gives
\[
  Q^\infty(\sigma_\alpha>j)
  \le \exp\left(-\frac{jg_0^2}{2\Delta_G^2}\right).
\]
The tail-sum calculation in Corollary~\ref{cor:testing-ess}, with \(g\) replaced by \(g_0\), yields constants \(A,B<\infty\) independent of \(Q\) and \(\alpha\), and it also gives the asserted uniform exponential tail.

Condition (iii) implies (iv), for example by taking \(\alpha=1/2\) and applying Markov's inequality.  Finally, assume (iv) and set \(B_k=\{\tau\le k\}\).  Then \(\sup_P P^k(B_k)\le\alpha\), while
\[
  \inf_{Q\in\Qcal}Q^k(B_k)
  =1-\sup_{Q\in\Qcal}Q^\infty(\tau>k)\longrightarrow1.
\]
For some \(k\), the latter quantity is strictly larger than \(\alpha\), proving (i).
\end{proof}

\begin{example}[Two observations can create testability]\label{ex:two-observations}
Let \(\X=[0,1]\), let \(\Pcal=\{\delta_x:x\in[0,1]\}\), and let \(\Qcal\) be all probability measures on \([0,1]\) that are not point masses.  No one-observation randomized test works, because
\[
  \sup_{P\in\Pcal}\E_P\varphi
  =\sup_{x\in[0,1]}\varphi(x)\ge\E_Q\varphi.
\]
For two observations, take \(B=\{(x_1,x_2):x_1\ne x_2\}\).  This is Borel because the diagonal is closed in \([0,1]^2\).  Every point-mass null has \(P^2(B)=0\), while \(Q^2(B)>0\) for every non-point-mass \(Q\).  Thus one degree-two certificate verifies condition~\textup{(iv)} of Theorem~\ref{thm:main}.

The separation is not uniform: for
\(Q_\varepsilon=(1-\varepsilon)\delta_0+\varepsilon\delta_1\),
\[
  Q_\varepsilon^2(B)=2\varepsilon(1-\varepsilon)\longrightarrow0.
\]
If \(\mathcal A_r=\{Q:Q^2(B)\ge1/r\}\), then
\(d_{\TV}(\mathscr C_2,\mathscr D_2(\mathcal A_r))\ge1/r\).  For atomless alternatives, by contrast, \(Q^2(B)=1\) uniformly.  This example is also a special case of Proposition~\ref{prop:weak-compact-degree-two} below.
\end{example}

The weak-$*$ closure in Corollary~\ref{cor:geometric-power-one} is not specific to change detection.  It is the compact representation of a uniform fixed-block gap; rational-margin refinement turns the pointwise inequalities of Theorem~\ref{thm:main} into countably many such gaps.

\begin{remark}[Why countability appears]\label{rem:countability}
The most direct reading of condition~\textup{(v)} is that a \emph{countable} family of fixed-sample tests, each merely beating its own null size on some alternatives, already upgrades to one power-one sequential procedure.  A stopping rule supplies only the countable family \(\{\tau\le n\}\), while countably many certificates can be assigned positive mixture weights.  Example~\ref{ex:countable-cocountable} shows that ``countable'' cannot be replaced by an arbitrary family.
\end{remark}

\begin{remark}[Uniform gaps amplify to near-perfect blocks]\label{rem:amplification}
Suppose a bounded block score \(h\) satisfies
\[
  \sup_{P\in\Pcal}\E_{P^n}h\le0,
  \qquad
  \inf_{Q\in\mathcal A}\E_{Q^n}h\ge\delta>0.
\]
Scale so that the range of \(h\) has length at most two.  Repeating it on \(k\) independent blocks and rejecting when the sum exceeds \(k\delta/2\), Hoeffding's inequality \cite{hoeffding1963} makes both worst-case null error and worst-case type-II error over \(\mathcal A\) at most \(\exp(-k\delta^2/8)\).  The difficulty in Theorem~\ref{thm:main} is therefore countable assembly, not amplification.
\end{remark}

\subsection{Countable assembly}

\begin{proposition}[A sufficient Lindel\"of assembly criterion]\label{prop:lindelof}
For every admissible finite-block score \((n,h)\), meaning that \(h\) is bounded and
\(\sup_{P\in\Pcal}\E_{P^n}h\le0\), set
\[
  U_{n,h}:=\{Q\in\Qcal:\E_{Q^n}h>0\}.
\]
Let \(\mathfrak T_{\Pcal}\) be the topology for which these sets form a subbase.  If every \(Q\in\Qcal\) is individually finite-block testable and \((\Qcal,\mathfrak T_{\Pcal})\) is Lindel\"of, then the equivalent conditions of Theorem~\ref{thm:main} hold.
\end{proposition}

\begin{proof}
The family of all \(U_{n,h}\)'s is an open cover, and Lindel\"ofness gives a countable subcover, which is condition~\textup{(vi)}.  This is only a sufficient topological criterion: the exact condition is that this particular subbasic cover admit a countable subcover, whereas Lindel\"ofness requires the same for every open cover.
\end{proof}

\begin{proposition}[Automatic assembly from continuous certificates]\label{prop:continuous-assembly}
Let \(\X\) be Polish with its Borel sigma-algebra.  Suppose that every \(Q\in\Qcal\) has some \(n\ge1\) and bounded continuous \(h:\X^n\to\R\) satisfying
\[
  \sup_{P\in\Pcal}\E_{P^n}h\le0<\E_{Q^n}h.
\]
Then the equivalent conditions of Theorem~\ref{thm:main} hold.
\end{proposition}

\begin{proof}
The weak topology on \(\M_1(\X)\) is second countable, hence every subspace is Lindel\"of \cite{billingsley1999}.  The product map \(R\mapsto R^n\) is weakly continuous: if \(R_j\Rightarrow R\), then \(R_j^n\Rightarrow R^n\).  Consequently, for bounded continuous \(h\), the map \(R\mapsto\E_{R^n}h\) is weakly continuous.  The pointwise positivity neighborhoods therefore form an open cover of \(\Qcal\) with a countable subcover.
\end{proof}

\subsection{Degree-two certificates}

\begin{lemma}[A general two-observation kernel certificate]\label{lem:kernel-certificate}
Let \(k:\X\times\X\to\R\) be bounded, measurable, symmetric, and positive definite.  For \(Q\in\M_1(\X)\), let \(Z,Z'\) be independent with common law \(Q\), and define the centered kernel
\[
  k_Q(x,y)
  :=k(x,y)-\E_{Z\sim Q}k(x,Z)-\E_{Z\sim Q}k(Z,y)
  +\E_{Z,Z'\sim Q}k(Z,Z'),
\]
and write
\[
  \operatorname{MMD}_k^2(P,Q):=\E_{P^2}k_Q,
\]
with \(\operatorname{MMD}_k\) denoting the nonnegative square root.  If
\[
  r_Q:=\inf_{P\in\Pcal}\operatorname{MMD}_k(P,Q)>0,
\]
then
\[
  h_Q(x,y):=r_Q^2-k_Q(x,y)
\]
is a bounded measurable symmetric two-observation certificate for \(Q\):
\[
  \sup_{P\in\Pcal}\E_{P^2}h_Q\le0,
  \qquad
  \E_{Q^2}h_Q=r_Q^2>0.
\]
If \(\X\) is metrizable and \(k\) is bounded and continuous, then \(h_Q\) is continuous.
\end{lemma}

\begin{proof}
Positive definiteness gives \(\operatorname{MMD}_k^2(P,Q)\ge0\); equivalently, the displayed quantity is the squared norm of the difference of the kernel mean embeddings, as in \cite{sriperumbudur-et-al-2010}.  Direct expansion yields
\[
  \E_{P^2}h_Q=r_Q^2-\operatorname{MMD}_k^2(P,Q),
  \qquad
  \E_{Q^2}h_Q=r_Q^2.
\]
The stated inequalities follow from the definition of \(r_Q\).  Now suppose \(\X\) is metrizable and \(k\) is bounded and continuous.  If \(x_j\to x\), then \(k(x_j,z)\to k(x,z)\) for every \(z\), and dominated convergence gives
\[
  \int k(x_j,z)\,dQ(z)\longrightarrow\int k(x,z)\,dQ(z).
\]
The analogous statement holds in the other coordinate.  Thus \(k_Q\), and hence \(h_Q\), is sequentially continuous; on a metrizable space this is equivalent to continuity.
\end{proof}

\begin{proposition}[Complete characterization on countable discrete sample spaces]\label{prop:countable-space-degree-two}
Suppose \(\X\) is countable and \(\B=2^\X\).  Identify a law \(R\) with its probability vector \(\theta_R=(R\{x\})_{x\in\X}\in\ell^1(\X)\), and put
\[
  \Theta_{\Pcal}:=\{\theta_P:P\in\Pcal\}.
\]
Then the equivalent conditions of Theorem~\ref{thm:main} hold if and only if
\[
  \theta_Q\notin\overline{\Theta_{\Pcal}}^{\,\ell^1}
  \qquad\text{for every }Q\in\Qcal.
\]
Whenever this holds, condition~\textup{(vi)} may be witnessed entirely by bounded symmetric two-observation scores.  Thus, on a countable discrete sample space, individual finite-block testability automatically assembles into one simultaneous power-one procedure.
\end{proposition}

\begin{proof}
If \(\theta_Q\in\overline{\Theta_{\Pcal}}^{\ell^1}\), choose \(P_j\in\Pcal\) with \(d_{\TV}(P_j,Q)\to0\).  The product bound
\[
  d_{\TV}(P_j^n,Q^n)\le n\,d_{\TV}(P_j,Q)
\]
follows, for example, by coupling each coordinate optimally and applying a union bound.  Hence, for every bounded \(h:\X^n\to\R\),
\[
  \left|\E_{P_j^n}h-\E_{Q^n}h\right|
  \le 2\lVert h\rVert_\infty d_{\TV}(P_j^n,Q^n)
  \le 2n\lVert h\rVert_\infty d_{\TV}(P_j,Q)\longrightarrow0.
\]
Hence no finite-block score can be nonpositive on all of \(\Pcal\) and strictly positive at \(Q\).  This proves necessity.

For sufficiency, first note that the \(\ell^1\)- and \(\ell^2\)-closures of \(\Theta_{\Pcal}\) agree inside the probability simplex.  The implication from \(\ell^1\) to \(\ell^2\) is immediate.  The finite case is immediate as well, so for the converse suppose \(\X=\{x_1,x_2,\ldots\}\) is countably infinite and \(p_j\to q\) in \(\ell^2\), where \(q\) is a fixed probability vector.  For every \(K\), writing \(d_j=p_j-q\) and using \(\sum_xd_j(x)=0\),
\[
  \lVert d_j\rVert_1
  =2\sum_x d_j(x)^-
  \le 2\sqrt K\,\lVert d_j\rVert_2
      +2\sum_{r>K}q(x_r).
\]
Letting first \(j\to\infty\) and then \(K\to\infty\) gives \(\lVert p_j-q\rVert_1\to0\).

Now fix \(Q\) outside the displayed closure and set
\[
  r_Q:=\inf_{P\in\Pcal}\lVert\theta_P-\theta_Q\rVert_2>0.
\]
For the counting kernel \(k(x,y)=\one\{x=y\}\),
\[
  \operatorname{MMD}_k^2(P,Q)
  =\lVert\theta_P-\theta_Q\rVert_2^2.
\]
Lemma~\ref{lem:kernel-certificate} therefore gives the explicit bounded score
\[
  h_Q(x,y)
  =r_Q^2-\lVert\theta_Q\rVert_2^2
   +\theta_Q(x)+\theta_Q(y)-\one\{x=y\}.
\]
The countable discrete space is Polish and this score is continuous on \(\X^2\).  Proposition~\ref{prop:continuous-assembly} supplies a countable subfamily of such certificates, proving simultaneous testability.  The finite-space result is the special case in which \(\ell^1\) and \(\ell^2\) are finite-dimensional norms.
\end{proof}

\begin{proposition}[Degree two for weakly compact nulls]\label{prop:weak-compact-degree-two}
Let \(\X\) be Polish with its Borel sigma-algebra and let \(\Pcal\subseteq\M_1(\X)\) be weakly compact.  Then every \(Q\notin\Pcal\) admits a bounded continuous symmetric two-observation score \(h_Q\) such that
\[
  \sup_{P\in\Pcal}\E_{P^2}h_Q\le0<\E_{Q^2}h_Q.
\]
Consequently, the equivalent conditions of Theorem~\ref{thm:main} hold for every \(\Qcal\subseteq\M_1(\X)\setminus\Pcal\), and all certificates may be taken of degree two.
\end{proposition}

\begin{proof}
The weak topology on \(\M_1(\X)\) is second countable and is the initial topology generated by the maps \(P\mapsto\E_Pf\), \(f\in C_b(\X)\) \cite{billingsley1999}.  To obtain a countable generating family, take a countable base of the weak topology.  Each base element is a union of finite-coordinate basic neighborhoods; second countability, hence Lindel\"ofness, supplies a countable such subcover.  Collecting the finitely many functions appearing in these neighborhoods over the countable base yields a countable family \((f_j)_{j\ge1}\subset C_b(\X)\) that generates the weak topology and therefore determines probability measures.  Normalize it so that \(|f_j|\le1\), and define
\[
  k(x,y):=\sum_{j\ge1}2^{-j}f_j(x)f_j(y).
\]
Uniform convergence makes \(k\) bounded and continuous, and its feature-map representation shows that it is positive definite.  Moreover,
\[
  \operatorname{MMD}_k^2(P,Q)
  =\sum_{j\ge1}2^{-j}
   \bigl(\E_Pf_j-\E_Qf_j\bigr)^2,
\]
so \(k\) is characteristic in the terminology of \cite{sriperumbudur-et-al-2010}.  The displayed series also shows that \(P\mapsto\operatorname{MMD}_k(P,Q)\) is weakly continuous.

Fix \(Q\notin\Pcal\).  Compactness is used only here: it gives
\[
  r_Q:=\inf_{P\in\Pcal}\operatorname{MMD}_k(P,Q)>0.
\]
Lemma~\ref{lem:kernel-certificate} supplies the desired continuous degree-two score.  The neighborhoods
\[
  U_Q:=\{R:\operatorname{MMD}_k(R,Q)<r_Q/2\}
\]
cover \(\M_1(\X)\setminus\Pcal\), and second countability supplies a countable subcover.  If \(R\in U_Q\), then
\[
  \frac12r_Q<r_R<\frac32r_Q,\qquad r_R:=\inf_{P\in\Pcal}\operatorname{MMD}_k(P,R),
  \qquad
  \E_{R^2}h_Q>\frac34r_Q^2,
\]
by the triangle inequality.  The construction also yields the following rate bookkeeping.  There is a universal \(H<\infty\) with \(\lVert h_Q\rVert_\infty\le H\) for all \(Q\).  For \(R\in U_Q\), put \(\mu=\E_{R^2}h_Q\) and choose a rational
\[
  \lambda\in\left[\frac{\mu}{4H^2},\frac{\mu}{2H^2}\right].
\]
Since \(0<\mu\le\lVert h_Q\rVert_\infty\le H\), this choice satisfies \(\lambda H\le\mu/(2H)\le1/2\), and
\[
  \E_{R^2}\log(1+\lambda h_Q)
  \ge\frac{\lambda\mu}{2}
  \ge\frac{\mu^2}{8H^2}\asymp r_R^4.
\]
Hence the coefficients of the logarithmic terms in the testing and change-detection bounds are \(O(r_R^{-4})\).  The additive terms additionally record the fixed mixture weight of the selected neighborhood.
\end{proof}

\begin{remark}[Relation to weak-compactness sufficiency]\label{rem:ram-ramdas}
Proposition~\ref{prop:weak-compact-degree-two} gives an independent short proof of the qualitative existence theorem of Ram and Ramdas \cite{ram-ramdas-2026} and strengthens that conclusion by showing that bounded continuous degree-two certificates suffice.  Their work additionally develops an asymptotically relatively growth-rate-optimal e-process.  Section~\ref{sec:quantitative-testing} characterizes this quantitative objective for general composite alternatives and recovers weak compactness as a sufficient condition.  The kernel lemma makes the division of labor transparent: compactness is used only to turn point separation in a characteristic metric into the positive distance \(r_Q>0\); the two-observation certificate itself is completely general.
\end{remark}

\begin{remark}[What failure of degree two would require]\label{rem:degree-two-obstruction}
By Proposition~\ref{prop:block-tv-identity} applied to the singleton \(\{Q\}\), and because \(\mathscr C_2\) is weak-$*$ closed, the absence of a degree-two certificate is equivalent to \(Q^2\in\mathscr C_2\).  It then follows that for every countable family \((f_j)_{j\ge1}\) of bounded measurable functions, there is a sequence \((P_i)\subseteq\Pcal\) such that
\[
  \E_{P_i}f_j\longrightarrow\E_Qf_j
  \qquad\text{for every fixed }j.
\]
Indeed, let
\[
  \widetilde f_j
  :=\frac{f_j-\E_Qf_j}{1+2\lVert f_j\rVert_\infty},
  \qquad
  k(x,y):=\sum_{j\ge1}2^{-j}\widetilde f_j(x)\widetilde f_j(y).
\]
Then \(k\) is a bounded measurable positive-definite kernel,
\[
  \E_{P^2}k
  =\sum_{j\ge1}2^{-j}\bigl(\E_P\widetilde f_j\bigr)^2\ge0,
  \qquad
  \E_{Q^2}k=0.
\]
Since the affine weak-$*$ continuous functional \(\nu\mapsto\E_\nu k\) has value zero at \(Q^2\in\mathscr C_2\), its infimum over the generators \(\{P^2:P\in\Pcal\}\) is zero.  Choosing \(P_i\) with \(\E_{P_i^2}k\le1/i\) proves the claim.

Thus any counterexample to collapse of the degree hierarchy at two must allow the null to approximate each relevant alternative simultaneously on every prescribed countable family of bounded observables.  This provides a quick obstruction for many natural moment and atom constraints.  It also explains the contrast with Example~\ref{ex:two-observations}: on an uncountable space, a discontinuous score such as the diagonal indicator may separate at degree two even when one-observation closed-convex separation fails.
\end{remark}

\begin{example}[Individual testability need not assemble countably]\label{ex:countable-cocountable}
Let \(\X\) be uncountable and let \(\B\) be the countable--cocountable sigma-algebra.  Define
\[
  P_0(A):=\begin{cases}
    0,&A\text{ countable},\\
    1,&A^c\text{ countable},
  \end{cases}
  \qquad
  \Pcal=\{P_0\},
  \qquad
  \Qcal=\{\delta_x:x\in\X\}.
\]
The set function \(P_0\) is a countably additive probability measure.  Each \(\delta_x\) is individually perfectly testable by \(\one_{\{x\}}\).  Nevertheless, no single power-one sequential test works for all of \(\Qcal\).

Indeed, for every \(n\) and every bounded \(\B^{\otimes n}\)-measurable \(\varphi\), there are a countable set \(S\subseteq\X\) and a constant \(c\) such that
\begin{equation}\label{eq:eventually-constant}
  \varphi\equiv c\quad\text{on }(\X\setminus S)^n.
\end{equation}
To verify this, let \(\mathcal H_n\) contain those sets \(A\subseteq\X^n\) for which some countable \(S\) makes \((\X\setminus S)^n\) either contained in or disjoint from \(A\).  This is a sigma-algebra containing every measurable rectangle, hence \(\B^{\otimes n}\subseteq\mathcal H_n\).  Applying this fact to the countably many rational superlevel sets of \(\varphi\) proves the display.

Now \(P_0^n((\X\setminus S)^n)=1\), so \(\E_{P_0^n}\varphi=c\); for every \(x\notin S\), also \(\E_{\delta_x^n}\varphi=\varphi(x,\ldots,x)=c\).  Thus one finite-block test strictly separates at most countably many of the alternatives, and a countable family separates at most countably many.  Condition~\textup{(v)} of Theorem~\ref{thm:main} therefore fails.  This shows that individual perfect testability does not imply simultaneous sequential testability on an arbitrary measurable space.  In the topology of Proposition~\ref{prop:lindelof}, the subspace \(\{\delta_x:x\in\X\}\) is uncountable and discrete, hence non-Lindel\"of.  It also shows that the Polish/Borel regularity in Proposition~\ref{prop:weak-compact-degree-two} cannot simply be removed, even when the null is a singleton.  Example~\ref{ex:lalley-nobel-rate-gap} below shows that individual finite-block testability can also fail to assemble countably on a standard Borel space, even with Borel hypothesis classes.
\end{example}

\subsection{Quantitative testing and simultaneous optimal growth}\label{sec:quantitative-testing}

For \(m\ge1\), let \(\mathcal G_m\) be the set of measurable functions
\(g:\X^m\to(0,\infty)\) such that
\[
  0<c_g\le g\le C_g<\infty,
  \qquad \sup_{P\in\Pcal}\E_{P^m}g\le1,
\]
where the constants may depend on \(g\).  Define
\begin{equation}\label{eq:intrinsic-growth-rate}
  G_m(Q):=\sup_{g\in\mathcal G_m}\E_{Q^m}\log g,
  \qquad
  \Gamma(Q):=\sup_{m\ge1}\frac{G_m(Q)}m
             =\lim_{m\to\infty}\frac{G_m(Q)}m,
  \qquad
  \Phi(Q):=\inf_{P\in\Pcal}\KL(Q\|P).
\end{equation}
The limit follows from superadditivity: products of admissible factors on disjoint blocks give
\(G_{m+n}(Q)\ge G_m(Q)+G_n(Q)\), and \(G_m(Q)\ge0\).
Flooring and truncating an e-variable shows that the bounded positive factors in
\(\mathcal G_m\) give the usual finite-sample growth-rate-optimal value.  More explicitly, if
\(e\ge0\) has null expectation at most one, then
\[
  g_{\varepsilon,L}:=\min\{(e+\varepsilon)/(1+\varepsilon),L\}
  \in\mathcal G_m
  \qquad(\varepsilon>0,\ L\ge1),
\]
and first letting \(L\to\infty\), then \(\varepsilon\downarrow0\), recovers every
well-defined expected log payoff of \(e\).  In particular, for such a payoff,
\begin{equation}\label{eq:finite-growth-upper}
  \E_{Q^m}\log e\le G_m(Q)\le m\Gamma(Q)\le m\Phi(Q).
\end{equation}
The last inequality is the entropy inequality applied to every \(P\in\Pcal\).
The intrinsic pointwise growth rate \(\Gamma\), including its reverse-information-projection
representation through the finite-sample effective nulls, was identified by Ram and Ramdas
\cite{ram-ramdas-optimal-growth-2026}.  The issue considered here is whether one procedure
attains these pointwise values simultaneously over \(\Qcal\).

For \(a,q\in[0,1]\), write
\begin{equation}\label{eq:one-sided-binary-kl}
  \operatorname{kl}_{+}(q,a):=
  \begin{cases}
    q\log(q/a)+(1-q)\log((1-q)/(1-a)),&q>a,\\
    0,&q\le a.
  \end{cases}
\end{equation}
Thus \(\operatorname{kl}_{+}(q,0)=\infty\) for \(q>0\), whereas
\(\operatorname{kl}_{+}(0,0)=\operatorname{kl}_{+}(q,1)=0\).
An event \(A\in\B^{\otimes m}\), with \(a=\pi_m(A)\), yields the rate profile
\begin{equation}\label{eq:event-growth-profile}
  I_{m,A}(Q):=\frac1m\operatorname{kl}_{+}(Q^m(A),a)\le\Gamma(Q).
\end{equation}
Indeed, when \(0<a<1\), the factors
\begin{equation}\label{eq:binary-betting-factor}
  g_{A,b}:=\frac ba\one_A+\frac{1-b}{1-a}\one_{A^c},
  \qquad b\in\mathbb Q\cap(a,1),
\end{equation}
belong to \(\mathcal G_m\).  Their expected log payoffs, together with the payoff zero of
the constant factor one, have supremum \(\operatorname{kl}_{+}(Q^m(A),a)\).
For \(a=0\), the factors \(1+L\one_A\), \(L\in\mathbb N\), give the same conclusion;
for \(a=1\), the rate is zero.

\subsubsection{A countable block construction}

\begin{lemma}[Bounded block factors and ordinary-time growth]\label{lem:bounded-block-growth}
Let \((m_j,g_j)_{j\ge1}\) be a countable family with \(g_j\in\mathcal G_{m_j}\), and put
\[
  r(Q):=\max\left\{0,\sup_{j\ge1}\frac1{m_j}\E_{Q^{m_j}}\log g_j\right\}.
\]
There exists an e-process \(\Eproc\) for \(\Pcal\), finite on every sample path at every
finite time and satisfying \(\Eproc_n\ge1/2\), such that, for every probability law \(Q\),
\begin{equation}\label{eq:block-envelope-as-growth}
  \frac1n\log\Eproc_n\longrightarrow r(Q)
  \quad Q^\infty\text{-almost surely},
  \qquad
  \liminf_{n\to\infty}\frac1n\E_{Q^\infty}\log\Eproc_n\ge r(Q).
\end{equation}
The convergence to \(+\infty\) is allowed.  If \(r(Q)=\Gamma(Q)\), the expected-log
limit also equals \(\Gamma(Q)\).

Moreover, positive constants \(\omega_j\), independent of \(\alpha\) and \(Q\), can be
chosen in the construction so that, for
\(\tau_\alpha:=\inf\{n\ge1:\Eproc_n\ge1/\alpha\}\), every factor with
\(\mu_j(Q):=\E_{Q^{m_j}}\log g_j>0\) gives
\begin{equation}\label{eq:block-expected-stopping-upper}
  \E_{Q^\infty}\tau_\alpha
  \le
  \frac{m_j}{\mu_j(Q)}
  \left\{
    \log\frac{2m_j}{\alpha\omega_j}
    +\max(0,\sup\log g_j)
  \right\}.
\end{equation}
\end{lemma}

\begin{proof}
For each \(j\), apply the offset construction in the proof of Theorem~\ref{thm:main}
to \(g_j\).  In its notation, let
\[
  \Pi^{j,s}_k:=\prod_{i=1}^k
  g_j(X_{s+(i-1)m_j+1},\ldots,X_{s+im_j}),
  \qquad 0\le s<m_j,
\]
and let
\begin{equation}\label{eq:quantitative-offset-process}
  V^j_n:=\frac1{m_j}
  \Pi^{j,\,n\bmod m_j}_{\lfloor n/m_j\rfloor}.
\end{equation}
For completeness, this construction retains \(\Pi^{j,s}_k\) only at the times
\(s+km_j\), assigns zero at other times, and averages over the \(m_j\) offsets.
For an ordinary-time bounded stopping time \(\sigma\), the first selected grid index is
a stopping time for \((\F_{s+km_j})_k\), with value \(\infty\) when \(\sigma\) misses
that grid.  Optional sampling of the nonnegative block-product supermartingale gives
expectation at most one for each grid contribution at \(\sigma\).  Averaging and
Lemma~\ref{lem:bounded-stopping-eprocess} show that \(V^j\) is an e-process for the full
filtration.  This argument does not retain an old block-product value while part of a new
block is being observed.

Set
\[
  B_j:=\max\{1,\|\log g_j\|_\infty\},\qquad
  N_j:=\left\lceil j^2(m_j+B_j)^4\right\rceil,
  \qquad \omega_j:=2^{-j}\exp(-N_jB_j),
\]
where the norm denotes the pointwise supremum, and define
\begin{equation}\label{eq:quantitative-mixture-process}
  \Eproc_n:=\frac12+\frac12\sum_{j\ge1}\omega_jV^j_n.
\end{equation}
The weights sum to at most one.  Tonelli's theorem at bounded stopping times and
Lemma~\ref{lem:bounded-stopping-eprocess} prove validity.  For fixed \(n\), a component
with \(N_j\ge n\) satisfies \(\omega_jV^j_n\le2^{-j}\), since
\(V^j_n\le\exp(nB_j)\).  Only finitely many indices have \(N_j<n\).  Thus
\eqref{eq:quantitative-mixture-process} is bounded uniformly over all sample paths at
each fixed finite time.

Fix \(Q\).  The strong law on each of the finitely many offsets gives
\[
  n^{-1}\log V^j_n\longrightarrow \mu_j(Q)/m_j
  \qquad Q^\infty\text{-almost surely}.
\]
Intersecting these probability-one events over \(j\) proves that the liminf in
\eqref{eq:block-envelope-as-growth} is at least \(r(Q)\).  This already proves the
assertion when \(r(Q)=\infty\).  Suppose \(r(Q)<\infty\), and fix \(\varepsilon>0\).
For \(k=\lfloor n/m_j\rfloor\ge1\),
\[
  \E_Q\log V^j_n=k\mu_j(Q)-\log m_j\le nr(Q).
\]
Hoeffding's inequality for the \(k\) independent, bounded log factors
\cite{hoeffding1963} gives
\[
  Q^\infty\{\log V^j_n>n(r(Q)+\varepsilon)\}
  \le\exp\left\{-\frac{nm_j\varepsilon^2}{2B_j^2}\right\}.
\]
When \(k=0\), the event on the left is empty.  The event
\(\{\omega_jV^j_n>2^{-j}e^{n(r(Q)+\varepsilon)}\}\) is also empty for \(n\le N_j\).
For the remaining indices its probabilities have summable total, because
\begin{align*}
 &\sum_{j\ge1}\sum_{n>N_j}
 Q^\infty\{\omega_jV^j_n>2^{-j}e^{n(r(Q)+\varepsilon)}\}\\
 &\qquad\le
 \sum_{j\ge1}\left(1+\frac{2B_j^2}{m_j\varepsilon^2}\right)
 \exp\left\{-\frac{N_jm_j\varepsilon^2}{2B_j^2}\right\}<\infty.
\end{align*}
The last assertion follows from \(B_j\ge1\) and
\(N_jm_j/B_j^2\ge j^2B_j^2\).
By Borel--Cantelli, eventually every component is at most
\(2^{-j}e^{n(r(Q)+\varepsilon)}\).  Summing proves the matching limsup bound.
Letting \(\varepsilon\) decrease through a countable sequence proves almost-sure
convergence.  Fatou's lemma applied to the nonnegative variables
\(n^{-1}\log(2\Eproc_n)\) gives the expected-log lower bound.  When
\(r(Q)=\Gamma(Q)\), \eqref{eq:finite-growth-upper} gives equality.

Finally, at times \(n=km_j\),
\(\Eproc_n\ge\omega_j\Pi^{j,0}_k/(2m_j)\).
Let \(K\) be the first time that the sum of the corresponding log factors exceeds
\(h=\log(2m_j/(\alpha\omega_j))\).  Positive drift and bounded increments imply
\(\E_QK<\infty\): Hoeffding's inequality bounds \(Q(K>k)\) by an exponentially
decreasing sequence once \(k\ge2h/\mu_j(Q)\).  Wald's identity and the upper bound on
the final increment give
\[
  \mu_j(Q)\E_QK
  \le h+\max(0,\sup\log g_j).
\]
Since \(\tau_\alpha\le m_jK\), this proves
\eqref{eq:block-expected-stopping-upper}.
\end{proof}

\subsubsection{Exact simultaneous characterization}

\begin{theorem}[Simultaneous attainment of the intrinsic growth rate]\label{thm:quantitative-testing}
The following statements are equivalent.
\begin{enumerate}[label=(\roman*)]
\item There is a countable family \((m_j,A_j)_{j\ge1}\), with
\(A_j\in\B^{\otimes m_j}\), such that
\begin{equation}\label{eq:binary-rate-recovery}
  \Gamma(Q)=\sup_{j\ge1}
  \frac{\operatorname{kl}_{+}(Q^{m_j}(A_j),\pi_{m_j}(A_j))}{m_j}
  \qquad\text{for every }Q\in\Qcal.
\end{equation}
\item There is a countable family \((m_j,g_j)_{j\ge1}\), with
\(g_j\in\mathcal G_{m_j}\), such that
\begin{equation}\label{eq:countable-rate-recovery}
  \Gamma(Q)=\sup_{j\ge1}\frac1{m_j}\E_{Q^{m_j}}\log g_j
  \qquad\text{for every }Q\in\Qcal.
\end{equation}
\item There is one e-process \(\Eproc\), finite everywhere at every finite time and
bounded below by \(1/2\), such that, for every \(Q\in\Qcal\),
\begin{equation}\label{eq:simultaneous-optimal-growth}
  \frac1n\log\Eproc_n\longrightarrow\Gamma(Q)
  \quad Q^\infty\text{-almost surely},
  \qquad
  \frac1n\E_{Q^\infty}\log\Eproc_n\longrightarrow\Gamma(Q).
\end{equation}
\item There is one e-process \(\Eproc\) for which
\[
  \liminf_{n\to\infty}\frac1n\E_{Q^\infty}\log\Eproc_n
  \ge\Gamma(Q)
  \qquad\text{for every }Q\in\Qcal.
\]
\end{enumerate}
Extended limits are allowed when \(\Gamma(Q)=\infty\).  On the alternatives with finite
\(\Gamma(Q)\), these conditions characterize zero pointwise asymptotic growth regret,
usually called asymptotic REGROW optimality.
\end{theorem}

\begin{proof}
The binary factors \eqref{eq:binary-betting-factor}, their zero-size counterparts, and
the constant factor one prove (i)$\Rightarrow$(ii).  Lemma~\ref{lem:bounded-block-growth}
proves (ii)$\Rightarrow$(iii), and (iii)$\Rightarrow$(iv) is immediate.

For (iv)$\Rightarrow$(ii), put \(F_n=(1+\Eproc_n)/2\) and, for integers \(L\ge1\),
put \(g_{n,L}=\min(F_n,L)\).  These are bounded positive finite-block e-variables.
Their expected logs increase to \(\E_Q\log F_n\) as \(L\to\infty\), while
\(\log F_n\ge\log\Eproc_n-\log2\).  For every \(Q\) and every real
\(r<\Gamma(Q)\), some \((n,L)\) therefore satisfies
\(n^{-1}\E_Q\log g_{n,L}>r\).  Including the constant factor one and using
\eqref{eq:finite-growth-upper} proves \eqref{eq:countable-rate-recovery}.

It remains to prove (iii)$\Rightarrow$(i).  Use all the events
\[
  A_{n,r}:=\{\Eproc_n\ge e^{nr}\},
  \qquad n\ge1,\quad r\in\mathbb Q_{>0}.
\]
Their null sizes satisfy \(a_{n,r}:=\pi_n(A_{n,r})\le e^{-nr}\).  If
\(0<r<\Gamma(Q)\), then \(q_{n,r}:=Q^n(A_{n,r})\to1\), so eventually
\(q_{n,r}>a_{n,r}\), and
\[
  \frac1n\operatorname{kl}_{+}(q_{n,r},a_{n,r})
  \ge q_{n,r}r-\frac{\log2}{n}.
\]
The supremum of these rates is at least every rational \(r<\Gamma(Q)\).
The reverse inequality is \eqref{eq:event-growth-profile}; a zero-rate event covers
\(\Gamma(Q)=0\).  This proves (i).
\end{proof}

\begin{proposition}[Sharp expected stopping times]\label{prop:sharp-power-one}
Let \(\tau\) be a level-\(\alpha\) sequential test, and suppose
\(L_Q:=\E_{Q^\infty}\tau<\infty\).  Then
\begin{equation}\label{eq:intrinsic-power-one-lower}
  \Gamma(Q)L_Q\ge\log(1/\alpha).
\end{equation}
If the conditions of Theorem~\ref{thm:quantitative-testing} hold, its common process
can be chosen so that, for every \(Q\) with \(0<\Gamma(Q)\le\infty\), its threshold
tests have finite expectation and satisfy
\begin{equation}\label{eq:optimal-power-one-sample-size}
  \lim_{\alpha\downarrow0}
  \frac{\E_{Q^\infty}\tau_\alpha}{\log(1/\alpha)}
  =\frac1{\Gamma(Q)}.
\end{equation}
Here \(1/\infty=0\).  Conversely, if \(\Gamma(Q)>0\) for every \(Q\in\Qcal\), a
family of level-\(\alpha\) tests attaining the upper bound in
\eqref{eq:optimal-power-one-sample-size} simultaneously implies the conditions of
Theorem~\ref{thm:quantitative-testing}.
\end{proposition}

\begin{proof}
Repeat \(\tau\) on successive fresh portions of the iid sequence, and let
\(S_k=\tau_1+\cdots+\tau_k\) be the time of the \(k\)th completion.  A null run fails
to complete with probability at least \(1-\alpha\), and independence at successive
finite stopping times gives
\[
  \sup_{P\in\Pcal}P^\infty(S_k\le N)\le\alpha^k.
\]
Under \(Q\), \(S_k/k\to L_Q\) almost surely.  For
\(N_k=\lceil k(L_Q+\varepsilon)\rceil\), the finite-block event
\(\{S_k\le N_k\}\) has \(Q\)-probability tending to one.  Applying
\eqref{eq:event-growth-profile}, and using the binary entropy bound \(\log2\), gives
\[
  \Gamma(Q)\ge\frac{\log(1/\alpha)}{L_Q+\varepsilon}.
\]
Letting \(\varepsilon\downarrow0\) proves \eqref{eq:intrinsic-power-one-lower}.

For the upper bound, choose the process from Lemma~\ref{lem:bounded-block-growth}.
For each fixed positive-drift factor, \eqref{eq:block-expected-stopping-upper} implies
\[
  \limsup_{\alpha\downarrow0}
  \frac{\E_Q\tau_\alpha}{\log(1/\alpha)}
  \le\frac{m_j}{\E_{Q^{m_j}}\log g_j}.
\]
Taking the infimum over the factors in \eqref{eq:countable-rate-recovery} proves
\eqref{eq:optimal-power-one-sample-size}, including the infinite-rate case.

For the converse, retain only levels \(\alpha_\ell=e^{-\ell}\).  For each \(\ell\),
include all events \(\{S_k^{(\ell)}\le N\}\), \(k,N\ge1\), obtained by repeating
\(\tau_{\alpha_\ell}\).  This family is countable and independent of \(Q\).
The first part of the proof shows that its rate supremum is at least
\(\ell/\E_Q\tau_{\alpha_\ell}\) whenever that expectation is finite.  The assumed
sample-size upper bound makes these ratios approach \(\Gamma(Q)\), or diverge when
\(\Gamma(Q)=\infty\).  Equation~\eqref{eq:event-growth-profile} supplies the reverse
inequality.  Thus \eqref{eq:binary-rate-recovery} holds.
\end{proof}

\subsubsection{Conditions ensuring simultaneous attainment}

\begin{proposition}[Separable alternatives and lower semicontinuous KL projections]\label{prop:quantitative-sufficient}
Each of the following conditions implies the countable recovery property in
Theorem~\ref{thm:quantitative-testing}.
\begin{enumerate}[label=(\alph*)]
\item The alternative class \(\Qcal\) is separable for total variation distance.
In particular, this holds on a Polish Borel sample space when all alternatives are
dominated by one sigma-finite measure.  No compactness or domination condition is
required of \(\Pcal\).
\item The sample space is Polish with its Borel sigma-algebra, and \(\Phi\) is weakly
lower semicontinuous at every \(Q\in\Qcal\).  In this case
\(\Gamma(Q)=\Phi(Q)\) on \(\Qcal\), and the common process attains \(\Phi\).
In particular, this conclusion holds for every weakly compact null class \(\Pcal\).
\end{enumerate}
\end{proposition}

\begin{proof}
For (a), boundedness of \(\log g\) and the inequality
\[
  d_{\TV}(Q^m,R^m)\le m\,d_{\TV}(Q,R)
\]
show that \(Q\mapsto m^{-1}\E_{Q^m}\log g\) is continuous for total variation distance.
For every positive rational \(r\), the sets
\[
  \{Q\in\Qcal:\E_{Q^m}\log g>mr\},
  \qquad m\ge1,\ g\in\mathcal G_m,
\]
cover \(\{Q\in\Qcal:\Gamma(Q)>r\}\).  A separable metric space is second countable,
as is each subspace, so this cover has a countable subcover.  Taking their union over
\(r\) and adjoining the constant factor one proves
\eqref{eq:countable-rate-recovery}.  Under the stated domination assumption, the
alternative densities belong to the separable space \(L^1\) of the common sigma-finite
measure; this gives total-variation separability.

For (b), fix \(Q\in\Qcal\) and \(0<r<\Phi(Q)\).  Lower semicontinuity at \(Q\)
provides a weak neighborhood on which \(\Phi>r\).  Finite intersections of constraints
on integrals of bounded continuous functions give a weakly closed convex set \(C\)
inside this neighborhood, with \(Q\) in its weak interior.  Csisz\'ar's nonasymptotic Sanov bound \cite{csiszar1984}, used in \cite{ram-ramdas-2026}, gives
\begin{equation}\label{eq:convex-sanov-quantitative}
  \sup_{P\in\Pcal}P^n\{\widehat Q_n\in C\}
  \le\exp\{-n\inf_{R\in C}\Phi(R)\}
  \le e^{-nr}.
\end{equation}
For every law \(R\) in the weak interior of \(C\), the almost-sure weak convergence
\(\widehat Q_n\Rightarrow R\) implies
\(R^n\{\widehat Q_n\in C\}\to1\).  Hence these events have rate supremum at least
\(r\) at each such \(R\).

For each positive rational \(r\), choose such a set around each
\(Q\in\Qcal\) with \(\Phi(Q)>r\).  Their weak interiors cover this subspace of the
second-countable weak law space.  A countable subcover, followed by a union over
\(r\) and \(n\), gives one countable event family whose rate supremum is at least
\(\Phi(Q)\) for every \(Q\in\Qcal\).  Equations~\eqref{eq:event-growth-profile} and
\eqref{eq:finite-growth-upper} then show that its supremum equals
\(\Gamma(Q)=\Phi(Q)\).

Finally, weak compactness of \(\Pcal\) implies weak lower semicontinuity of \(\Phi\).
Indeed, if \(R_k\Rightarrow R\) and \(\liminf_k\Phi(R_k)<\infty\), choose a subsequence
attaining this liminf and \(P_k\in\Pcal\) with
\(\KL(R_k\|P_k)\le\Phi(R_k)+1/k\).  A further subsequence converges weakly to
\(P\in\Pcal\), and joint weak lower semicontinuity of KL gives
\(\Phi(R)\le\KL(R\|P)\le\liminf_k\Phi(R_k)\).
\end{proof}

\begin{corollary}[Attainment of the raw KL projection]\label{cor:raw-kl-attainment}
A common e-process can attain \(\Phi(Q)\) in both senses of
\eqref{eq:simultaneous-optimal-growth} for every \(Q\in\Qcal\) if and only if
\(\Gamma(Q)=\Phi(Q)\) throughout \(\Qcal\) and the countable recovery property
\eqref{eq:countable-rate-recovery} holds.
\end{corollary}

\begin{proof}
Sufficiency is Theorem~\ref{thm:quantitative-testing}.  Necessity of
\(\Gamma=\Phi\) follows from \eqref{eq:finite-growth-upper}, and the theorem then
gives countable recovery.
\end{proof}

\subsection{Attainable rates and the cost of simultaneous optimality}\label{sec:rate-profiles}

For an e-process with a fixed cash reserve, \(\Eproc_n\ge1/2\), write
\[
  g_{\Eproc}(Q):=\liminf_{n\to\infty}
       \frac1n\E_{Q^n}\log\Eproc_n.
\]
Then \(0\le g_{\Eproc}(Q)\le\Gamma(Q)\).  A nonnegative function \(r\) on \(\Qcal\) is an \emph{attainable lower rate profile} if some such process has \(g_{\Eproc}(Q)\ge r(Q)\) for every \(Q\).  The next result also describes what remains attainable when \eqref{eq:countable-rate-recovery} fails.

\begin{proposition}[Rate profiles, deficits, and scalar optima]\label{prop:rate-profiles}
\begin{enumerate}[label=(\roman*),leftmargin=*]
\item A nonnegative profile \(r\) is attainable if and only if there is a countable collection of bounded block e-variables \((m_j,f_j)\), each bounded away from zero, such that
\begin{equation}\label{eq:attainable-rate-profile}
  r(Q)\le\sup_j\frac1{m_j}\E_{Q^{m_j}}\log f_j
  \qquad(Q\in\Qcal).
\end{equation}
The realizing cash-hedged process may also satisfy
\(\liminf_n n^{-1}\log\Eproc_n\ge r(Q)\), \(Q^\infty\)-almost surely, for every \(Q\).  Attainable lower profiles are closed under countable pointwise suprema.

\item Suppose \(0<\Gamma(Q)<\infty\) for every \(Q\).  If the oracle profile \(\Gamma\) is not attainable, then every cash-hedged process has
\[
  \{Q:g_{\Eproc}(Q)<\Gamma(Q)\}
  \quad\text{uncountable}.
\]
If, more strongly, no process has a positive rate at every alternative, then every such process has \(\{Q:g_{\Eproc}(Q)=0\}\) uncountable.  The first conclusion alone does not imply the second.

\item Under the assumptions of \textup{(ii)}, both scalar values
\begin{equation}\label{eq:best-fraction-additive-deficit}
  \rho_*:=\sup_{\Eproc}\inf_{Q\in\Qcal}
                \frac{g_{\Eproc}(Q)}{\Gamma(Q)},
  \qquad
  \Delta_*:=\inf_{\Eproc}\sup_{Q\in\Qcal}
                \{\Gamma(Q)-g_{\Eproc}(Q)\}
\end{equation}
are attained, with the second understood in \([0,\infty]\).  If \(\Gamma\) is unattainable, then \(\rho_*<1\) and \(\Delta_*>0\).  For every \(\rho\in(\rho_*,1]\), every process has \(g_{\Eproc}(Q)<\rho\Gamma(Q)\) at uncountably many alternatives; for every \(d\in[0,\Delta_*)\), it has \(\Gamma(Q)-g_{\Eproc}(Q)>d\) at uncountably many alternatives.  Nevertheless, if \(\Gamma\) is unattainable, there is no undominated rate profile: every \(g_{\Eproc}\) is weakly improved everywhere and strictly improved somewhere by another cash-hedged e-process.
\end{enumerate}
\end{proposition}

\begin{proof}
For necessity in \textup{(i)}, let \(\Eproc\) attain \(r\) and enumerate
\(f_{n,K}:=\Eproc_n\wedge K\), \(n,K\ge1\).  These are block e-variables in \([1/2,K]\), and monotone convergence after adding \(\log2\) gives
\[
  \sup_{n,K}\frac1n\E_{Q^n}\log f_{n,K}
  \ge\liminf_n\frac1n\E_{Q^n}\log\Eproc_n
  \ge r(Q).
\]
For sufficiency, apply the construction in Lemma~\ref{lem:bounded-block-growth} to the given factors, including the constant factor one.  Its almost-sure lower growth bound is the right side of \eqref{eq:attainable-rate-profile}.  Cash hedging and Fatou's lemma give the same expected lower bound.  Finally, for weights \(w_j>0\) summing to one, a countable mixture of processes preserves the cash floor and satisfies
\begin{equation}\label{eq:rate-mixture-dominance}
  g_{\sum_jw_j\Eproc^{(j)}}(Q)
  \ge\sup_j g_{\Eproc^{(j)}}(Q),
\end{equation}
because its logarithm is at least \(\log w_j+\log\Eproc_n^{(j)}\) for each fixed \(j\).

For each fixed \(Q\), the point-alternative result of \cite{ram-ramdas-optimal-growth-2026}, or Theorem~\ref{thm:quantitative-testing} applied to \(\{Q\}\), supplies a cash-hedged process with rate \(\Gamma(Q)\) there.  If a process had only countably many oracle deficits, mixing it with these individual oracle processes would attain \(\Gamma\) everywhere, by \eqref{eq:rate-mixture-dominance}.  The same argument patches a countable zero set and proves \textup{(ii)}.

Choose processes whose worst relative rates tend to \(\rho_*\).  Their positive countable mixture has worst relative rate at least \(\rho_*\), hence exactly \(\rho_*\).  If \(\Delta_*<\infty\), mix processes whose worst additive deficits decrease to \(\Delta_*\); the mixture has deficit at most \(\Delta_*\) at every alternative and therefore attains the infimum.  If \(\Delta_*=\infty\), every process attains that extended value.  Attainment shows that \(\rho_*=1\) or \(\Delta_*=0\) would make \(\Gamma\) attainable.  If either stated deficit set were countable, patching it with the corresponding pointwise oracle processes would contradict the definition of \(\rho_*\) or \(\Delta_*\).  Finally, whenever \(g_{\Eproc}(Q_0)<\Gamma(Q_0)\), mixing \(\Eproc\) with an individual oracle process for \(Q_0\) weakly improves every rate and strictly improves the rate at \(Q_0\).  This proves \textup{(iii)}.
\end{proof}

\begin{corollary}[Testing at attainable rate profiles]\label{cor:testing-rate-profiles}
Let \(0<r(Q)\le\Gamma(Q)\) for every \(Q\in\Qcal\), allowing infinite values.  The profile \(r\) is attainable if and only if there is a family of level-\(\alpha\) power-one tests satisfying
\[
  \limsup_{\alpha\downarrow0}
  \frac{\E_{Q^\infty}\tau_\alpha}{\log(1/\alpha)}
  \le\frac1{r(Q)}\qquad(Q\in\Qcal).
\]
When this holds, the tests may be the threshold family of one e-process.
\end{corollary}

\begin{proof}
Under \eqref{eq:attainable-rate-profile}, Lemma~\ref{lem:bounded-block-growth} supplies the common process, and \eqref{eq:block-expected-stopping-upper} gives the stated bound.  Conversely, repeat the tests at levels \(e^{-\ell}\), \(\ell\in\mathbb N\), and enumerate their finite-block completion events as in the converse proof of Proposition~\ref{prop:sharp-power-one}.  Their rate envelope is at least \(\ell/\E_Q\tau_{e^{-\ell}}\), hence at least \(r(Q)\) after taking the supremum over \(\ell\).  The corresponding binary bets and Proposition~\ref{prop:rate-profiles} give attainability.
\end{proof}

\begin{example}[A natural Borel model with a strict oracle gap]\label{ex:lalley-nobel-rate-gap}
Let \(\lambda\) be uniform measure on \([0,1]\), and let \(s(P)\) denote the total mass of the singular part of \(P\) relative to \(\lambda\).  Fix \(d>0\), put \(\delta=e^{-d}\), and choose full-support Bernoulli laws \(B_0,B_1\) with
\(c=\KL(B_1\|B_0)>0\).  On the compact Polish sample space \([0,1]\times\{0,1\}\), consider
\[
  \Pcal=\{P\otimes B_0:s(P)\le\delta\},
  \qquad
  \Qcal=\{Q\otimes B_1:Q\perp\lambda\}.
\]
These are Borel classes of laws.  For every alternative \(Q\otimes B_1\),
\begin{equation}\label{eq:lalley-nobel-oracle-gap}
  \Gamma(Q\otimes B_1)
  =\inf_{R\in\Pcal}\KL(Q\otimes B_1\|R)=c+d.
\end{equation}
The Bernoulli likelihood ratio achieves rate \(c\) simultaneously at all alternatives, and this is the largest attainable worst-case rate.  In fact, every cash-hedged e-process has rate at most \(c\) at uncountably many alternatives.  Consequently,
\[
  \rho_*=\frac{c}{c+d},\qquad \Delta_*=d.
\]
Deleting the Bernoulli coordinate gives individually testable alternatives with common oracle rate \(d\), but no simultaneous power-one test.  Thus individual testability need not assemble countably even on a standard Borel space with Borel hypothesis classes.
\end{example}

\begin{proof}
For Borel measurability, let \(\mathscr O\) be the countable collection of finite unions of rational open intervals, relative to \([0,1]\).  Regularity and absolute continuity give
\[
  s(P)=\inf_{r\in\mathbb N}\sup\{P(O):O\in\mathscr O,\ \lambda(O)<1/r\},
\]
so \(s\) is Borel as a function of the law.

Fix \(Q\perp\lambda\) and a Borel set \(S_Q\) with \(Q(S_Q)=1\) and \(\lambda(S_Q)=0\).  Write \(L(y)=dB_1/dB_0(y)\).  The bounded one-observation e-variable
\[
  f_Q(x,y)=\delta^{-1}\one_{S_Q}(x)L(y)
\]
has expected log \(c+d\) under \(Q\otimes B_1\).  Conversely,
\(R_Q=(\delta Q+(1-\delta)\lambda)\otimes B_0\) belongs to \(\Pcal\) and has \(\KL(Q\otimes B_1\|R_Q)=c+d\).  The information upper bound and the repeated factor \(f_Q\) prove \eqref{eq:lalley-nobel-oracle-gap}.  The cash-hedged likelihood ratio
\(\{1+\prod_{t=1}^n L(Y_t)\}/2\) has rate \(c\) at every alternative.

Suppose a process had expected lower rate strictly greater than \(c\) at every alternative.  Proposition~\ref{prop:rate-profiles}\textup{(i)} would give a process \(\Eproc\) with the same strict almost-sure lower rate bounds.  Put
\[
  L_n^Y:=\prod_{t=1}^n L(Y_t),\qquad F_n:=\Eproc_n/L_n^Y.
\]
For every \(P\) with \(s(P)\le\delta\) and every bounded stopping time \(\sigma\), stopped change of measure gives
\[
  \E_{(P\otimes B_1)^\infty}F_\sigma
  =\E_{(P\otimes B_0)^\infty}\Eproc_\sigma\le1.
\]
Thus \(F\), with its liminf terminal value, is an e-process for these modified null laws.  Under every \(Q\otimes B_1\) with \(Q\perp\lambda\), the strong law for \(L_n^Y\) gives \(\liminf_n n^{-1}\log F_n>0\).  Hence the decisions
\(\psi_n=\one\{F_n\ge n^2\}\) converge almost surely to one under every singular \(Q\otimes B_1\).  Under every absolutely continuous \(P\otimes B_1\), Markov's inequality and Borel--Cantelli make them converge almost surely to zero.

Integrate out \((Y_1,\ldots,Y_n)\sim B_1^n\) and threshold the resulting function of \((X_1,\ldots,X_n)\) at \(1/2\).  Fubini and dominated convergence show that these nonrandomized decisions are almost surely consistent for distinguishing all absolutely continuous laws from all singular laws.  This contradicts \cite[Theorem~1]{lalley-nobel-2003}.  Thus no common process can have rate strictly greater than \(c\) everywhere.  If the set on which a given process has rate at most \(c\) were countable, mixing with individual oracle processes on that set would give precisely the forbidden strict improvement everywhere.  This proves the uncountability assertion and both scalar values.

The same argument without the Bernoulli coordinate rules out simultaneous positive growth for the null \(\{P:s(P)\le\delta\}\) against all singular laws.  The constructive implication of Theorem~\ref{thm:main} would turn a simultaneous power-one test into positive growth, so such a test cannot exist.  Each individual alternative still has the factor \(\delta^{-1}\one_{S_Q}\) and oracle rate \(d\).
\end{proof}

\begin{remark}[Fixed-alternative growth and local power]\label{rem:janssen-local-power}
Janssen's finite-dimensional restrictions on goodness-of-fit power \cite{janssen-2000} concern fixed-level power in Gaussian shift experiments and, in the i.i.d. formulation, alternatives at distance of order \(n^{-1/2}\).  They do not preclude optimal growth at every fixed alternative: for the Gaussian sequence model \(P_h=\bigotimes_{k\ge1}N(h_k,1)\), \(h\in\ell^2\), on \(\mathbb R^{\mathbb N}\), one has \(\E_{P_h}\log(dP_g/dP_0)=\|h\|^2/2-\|h-g\|^2/2\), so a positive likelihood-ratio mixture over a countable dense set of shifts attains the oracle rate at each fixed shift.  The present obstruction concerns simultaneous fixed-alternative growth: Example~\ref{ex:lalley-nobel-rate-gap} makes it impossible even though all oracle rates equal the same positive constant.
\end{remark}

\section{Change detection with fixed-horizon and optional-horizon ARL control}
\label{sec:change-detection}

We now consider an unknown changepoint separating an i.i.d. pre-change law from an i.i.d. post-change law.  The objective is to control false alarms at an average-run-length (ARL) scale while detecting every fixed post-change distribution substantially faster than that scale.  Bare ARL control is only a mean constraint and permits front-loaded false alarms.  We supplement it with fixed-time false-alarm bounds.  These suffice for all converse results below, while our constructions satisfy the stronger optional-horizon inequality and are generated by one common e-detector for the full threshold family.  The distinction from ordinary ARL remains substantive even when all null-law means are finite and tightly calibrated.

\subsection{Model, false-alarm control, e-detectors, and delay}

Throughout this section, \((\F_t)\) denotes the actual filtration used by the procedure.  Without auxiliary randomization it is the coordinate filtration from Section~\ref{sec:setup}; when randomization is used, it is enlarged to include the auxiliary variables revealed through time \(t\).  We allow only auxiliary randomization that is independent of the observations under every law, and all probabilities and expectations below include it.  For \(P\in\Pcal\), \(Q\in\Qcal\), and an integer \(\nu\ge1\), let \(\Prob_{\nu}^{P,Q}\) denote the law whose observation marginal is
\[
  P^{\nu-1}\otimes Q^\infty,
\]
so that \(X_1,\ldots,X_{\nu-1}\) are i.i.d. \(P\) and \(X_\nu,X_{\nu+1},\ldots\) are i.i.d. \(Q\); when auxiliary randomization is present, its independent law is included as an additional product factor.  We write \(\E_{\nu}^{P,Q}\) for the corresponding expectation.  The no-change law is \(P^\infty\).

\begin{definition}[E-detector and threshold rule]\label{def:e-detector}
An adapted process \(\Detector=(\Detector_t)_{t\ge0}\) with values in \([0,\infty]\) and \(\Detector_0=0\) is an \emph{e-detector for \(\Pcal\)} if
\begin{equation}\label{eq:e-detector-inequality}
  \E_{P^\infty}\Detector_\tau\le \E_{P^\infty}\tau
  \qquad
  \text{for every }P\in\Pcal
  \text{ and every stopping time }\tau.
\end{equation}
Here \(\Detector_\infty:=\liminf_t\Detector_t\).  If \(\E_{P^\infty}\tau=\infty\), the inequality is interpreted in the extended sense and is automatic.  For \(\gamma\ge1\), define
\begin{equation}\label{eq:e-detector-threshold}
  T_\gamma:=\inf\{t\ge1:\Detector_t\ge\gamma\}.
\end{equation}
The robust average run length is
\[
  \operatorname{ARL}_{\Pcal}(T)
  :=\inf_{P\in\Pcal}\E_{P^\infty}T.
\]
\end{definition}

\begin{lemma}[Bounded stopping times suffice for e-detectors]\label{lem:bounded-stopping-edetector}
Let \(\Detector\) be nonnegative and adapted with
\(\Detector_\infty=\liminf_t\Detector_t\).  If
\[
  \E_{P^\infty}\Detector_\sigma\le\E_{P^\infty}\sigma
\]
for every \(P\in\Pcal\) and every bounded stopping time \(\sigma\), then \(\Detector\) is an e-detector.
\end{lemma}

\begin{proof}
Apply Fatou's lemma to \(\Detector_{\tau\wedge n}\) and monotone convergence to \(\tau\wedge n\):
\[
  \E_{P^\infty}\Detector_\tau
  \le\liminf_{n\to\infty}\E_{P^\infty}\Detector_{\tau\wedge n}
  \le\lim_{n\to\infty}\E_{P^\infty}(\tau\wedge n)
  =\E_{P^\infty}\tau.
\]
\end{proof}

\begin{definition}[Optional-horizon ARL validity]\label{def:optional-arl}
Fix \(\gamma\ge1\).  An alarm time is an \((\F_t)\)-stopping time taking values in \(\{1,2,\ldots\}\cup\{\infty\}\).  Such an alarm time \(T\) is \emph{optionally ARL-valid at scale \(\gamma\)} for \(\Pcal\) if
\begin{equation}\label{eq:optional-horizon-arl}
  \Prob_{P^\infty}(T\le\sigma)
  \le \frac{\E_{P^\infty}\sigma}{\gamma}
  \qquad
  \text{for every }P\in\Pcal
  \text{ and every stopping time }\sigma.
\end{equation}
Here \(\{T\le\sigma\}\) means \(\{T<\infty,\ T\le\sigma\}\).  A family \((T_\gamma)_{\gamma\ge1}\) is \emph{optionally ARL-valid} if \(T_\gamma\) satisfies \eqref{eq:optional-horizon-arl} at scale \(\gamma\) for every \(\gamma\ge1\).
\end{definition}

\begin{definition}[Fixed-horizon ARL validity]\label{def:fixed-arl}
An alarm \(T\) is \emph{fixed-horizon ARL-valid at scale \(\gamma\)} if
\(\inf_{P\in\Pcal}\E_{P^\infty}T\ge\gamma\) and
\begin{equation}\label{eq:fixed-horizon-alarm}
  \sup_{P\in\Pcal}\Prob_{P^\infty}(T\le n)\le n/\gamma
  \qquad\text{for every deterministic integer }n\ge1.
\end{equation}
A family is fixed-horizon ARL-valid if this holds at every scale.  We reserve
\emph{ordinary ARL control} for the mean constraint alone.
\end{definition}

Optional-horizon validity implies fixed-horizon ARL validity, by taking
\(\sigma=n\) and, when the null mean is finite, \(\sigma=T\).
The deterministic inequality alone gives only
\[
  \E_{P^\infty}T\ge\sum_{n\ge0}(1-n/\gamma)^+\ge\gamma/2;
\]
for example, a uniform alarm on \(\{1,\ldots,\gamma\}\) at an integer scale
has mean \((\gamma+1)/2\).  This is why the ordinary ARL requirement is stated
separately.  Every converse below uses only \eqref{eq:fixed-horizon-alarm}.

\begin{proposition}[Fixed-scale universality of e-detectors]\label{prop:edetector-universality}
Fix \(\gamma\ge1\) and an alarm time \(T\).  The following are equivalent:
\begin{enumerate}[label=(\roman*)]
\item \(T\) is optionally ARL-valid at scale \(\gamma\);
\item there exists an e-detector \(\Detector\) for \(\Pcal\) such that
\[
  T=\inf\{t\ge1:\Detector_t\ge\gamma\}.
\]
\end{enumerate}
Under \textup{(i)}, the canonical witness is \(\Detector_t=\gamma\one\{T\le t\}\).
\end{proposition}

\begin{proof}
Assume \textup{(ii)}, fix a stopping time \(\sigma\), and put \(\tau=T\wedge\sigma\).  On \(\{T\le\sigma\}\), one has \(\Detector_\tau=\Detector_T\ge\gamma\).  Hence
\[
  \gamma\Prob_{P^\infty}(T\le\sigma)
  \le \E_{P^\infty}\Detector_\tau
  \le \E_{P^\infty}\tau
  \le \E_{P^\infty}\sigma,
\]
which is \textup{(i)}.  Conversely, assume \textup{(i)} and set \(\Detector_t=\gamma\one\{T\le t\}\), with \(\Detector_0=0\).  For every stopping time \(\tau\), the convention in Definition~\ref{def:optional-arl} gives
\[
  \E_{P^\infty}\Detector_\tau
  =\gamma\Prob_{P^\infty}(T\le\tau)
  \le\E_{P^\infty}\tau.
\]
Thus \(\Detector\) is an e-detector and its level-\(\gamma\) crossing is exactly \(T\).  This is the canonical fixed-scale witness from the universality theorem in \cite{ramdas-edetector-universality-2026}.
\end{proof}

Every optionally ARL-valid family satisfies
\begin{equation}\label{eq:arl-lower}
  \operatorname{ARL}_{\Pcal}(T_\gamma)\ge\gamma.
\end{equation}
Indeed, fix \(P\in\Pcal\).  If \(\E_{P^\infty}T_\gamma=\infty\), there is nothing to prove.  Otherwise \(T_\gamma<\infty\) almost surely, and applying \eqref{eq:optional-horizon-arl} with \(T=T_\gamma\), scale \(\gamma\), and \(\sigma=T_\gamma\) gives \(1\le\E_{P^\infty}T_\gamma/\gamma\).

The equivalence in Proposition~\ref{prop:edetector-universality} is scale-by-scale: it does not assert that an arbitrary optionally ARL-valid family is generated by one common process.  The constructive direction of Theorem~\ref{thm:change-detection} below supplies that stronger conclusion.  Ramdas \cite{ramdas-edetector-universality-2026} calls the original notion in Definition~\ref{def:e-detector} a \emph{strong} e-detector only to distinguish it from a weak e-detector, which is universal for bare ARL control.  We retain the established unmodified term \emph{e-detector}.  The converse below uses only the weaker fixed-time bound, while e-detector representability characterizes the stronger optional-horizon property.

For a fixed post-change law \(Q\), use the robust Lorden delay
\begin{equation}\label{eq:lorden-delay}
  \mathcal J_Q(T)
  :=
  \sup_{P\in\Pcal}\sup_{\nu\ge1}
  \operatorname*{ess\,sup}
  \E_{\nu}^{P,Q}
  \left[(T-\nu+1)^+\mid\F_{\nu-1}\right].
\end{equation}
This controls the worst conditional expected delay over the pre-change law, changepoint, and pre-change history.  We call the pair \((\Pcal,\Qcal)\) \emph{pointwise change-detectable under fixed-horizon ARL control} if there is a fixed-horizon ARL-valid family \((T_\gamma)_{\gamma\ge1}\) satisfying
\[
  \mathcal J_Q(T_\gamma)/\gamma\longrightarrow0
  \qquad\text{for every fixed }Q\in\Qcal.
\]
The theorem below proves that the finite-block condition is exactly equivalent to this operational notion and, in addition, constructs one common e-detector satisfying optional-horizon validity.  Proposition~\ref{prop:edetector-universality} applies to that stronger validity property, rather than to fixed-horizon validity alone.  Corollary~\ref{cor:geometric-power-one} gives the equivalent weak-$*$ closed-hull formulation of the finite-block condition.

\subsection{Complete characterization}

\begin{theorem}[Power-one testing and i.i.d. change detectability]
\label{thm:change-detection}
The following statements are equivalent.
\begin{enumerate}[label=(\roman*)]
\item The equivalent conditions of Theorem~\ref{thm:main} hold.

\item For every \(\eta>0\), there exists an optionally ARL-valid family \((T_\gamma)_{\gamma\ge1}\) satisfying
\begin{equation}\label{eq:calibrated-arl}
  \gamma
  \le \inf_{P\in\Pcal}\E_{P^\infty}T_\gamma
  \le \sup_{P\in\Pcal}\E_{P^\infty}T_\gamma
  \le (1+\eta)\gamma+1,
  \qquad \gamma\ge1,
\end{equation}
and, for every fixed \(Q\in\Qcal\), there are finite constants \(A_{Q,\eta}\) and \(B_Q\) such that
\begin{equation}\label{eq:log-delay}
  \mathcal J_Q(T_\gamma)
  \le A_{Q,\eta}+B_Q\log\gamma,
  \qquad \gamma\ge2.
\end{equation}

\item There exists a fixed-horizon ARL-valid family \((T_\gamma)_{\gamma\ge1}\) such that
\begin{equation}\label{eq:sublinear-delay}
  \frac{\mathcal J_Q(T_\gamma)}{\gamma}\longrightarrow0
  \qquad\text{as }\gamma\to\infty
  \quad\text{for every fixed }Q\in\Qcal.
\end{equation}

\item There exists a fixed-horizon ARL-valid family \((T_\gamma)_{\gamma\ge1}\) such that, for every \(Q\in\Qcal\), some integer \(\gamma\ge3\) satisfies
\begin{equation}\label{eq:one-scale-delay}
  \mathcal J_Q(T_\gamma)<\frac\gamma2.
\end{equation}
\end{enumerate}
Whenever these conditions hold, for every \(\eta>0\) the family in \textup{(ii)} may be chosen as the full threshold family of one e-detector.  Thus the common-process representation is an additional conclusion, not part of the operational definition.  Consequently, one may choose a family for which every null-law expected run length is finite and
\[
  \mathcal J_Q(T_\gamma)
  <\operatorname{ARL}_{\Pcal}(T_\gamma)
\]
for every fixed \(Q\) and all sufficiently large \(\gamma\), where the threshold from which the inequality holds may depend on \(Q\).  By Corollary~\ref{cor:geometric-power-one}, these conditions are also equivalent to \eqref{eq:countable-tv-cover}.
\end{theorem}

\begin{proof}
We first prove (i)\(\Rightarrow\)(ii).  Fix \(\eta>0\).  On a product extension, let \((U_t)_{t\ge1}\) be i.i.d. \(\operatorname{Exp}(1)\), independent of the observations, and in this implication take
\[
  \F_t:=\sigma(X_1,\ldots,X_t,U_1,\ldots,U_t).
\]
Let \((n_m,h_m)\) satisfy condition~\textup{(vi)} of Theorem~\ref{thm:main}.  Put
\(a_m=\inf h_m\), \(H_m=\lVert h_m\rVert_\infty\), and
\[
  \mathcal K_m:=
  \begin{cases}
    \mathbb Q\cap(0,-1/a_m),&a_m<0,\\
    \mathbb Q\cap(0,\infty),&a_m\ge0.
  \end{cases}
\]
\emph{Restarted detector.}
Index the pairs \((m,\lambda)\), \(\lambda\in\mathcal K_m\), by \(\iota\in\mathcal I\).  Write
\[
  N_\iota=n_m,
  \qquad
  G_\iota=1+\lambda h_m.
\]
Then \(G_\iota\) is bounded and bounded away from zero, and
\begin{equation}\label{eq:block-e-variable}
  \sup_{P\in\Pcal}\E_{P^{N_\iota}}G_\iota\le1.
\end{equation}
Choose \(w_\iota>0\) with \(\sum_{\iota\in\mathcal I}w_\iota=1\).

For a candidate changepoint \(k\ge1\), define
\[
  \Pi_j^{k,\iota}
  :=\prod_{q=1}^j
  G_\iota\bigl(
    X_{k+(q-1)N_\iota},\ldots,X_{k+qN_\iota-1}
  \bigr),
  \qquad j\ge1,
\]
and
\begin{equation}\label{eq:sparse-restart}
  \Xi_t^{k,\iota}:=
  \begin{cases}
    \Pi_j^{k,\iota},&t=k-1+jN_\iota\text{ for some }j\ge1,\\
    0,&\text{otherwise.}
  \end{cases}
\end{equation}
Conditional on \(\F_{k-1}\), \((\Pi_j^{k,\iota})_{j\ge0}\), with \(\Pi_0^{k,\iota}=1\), is a nonnegative supermartingale on the enlarged block grid \((\F_{k-1+jN_\iota})_{j\ge0}\); the independent variables \((U_t)\) do not change the conditional means of the observation-block factors.  If \(\sigma\) is bounded, define \(J=j\) on \(\{\sigma=k-1+jN_\iota\}\), \(j\ge1\), and \(J=\infty\) otherwise.  This is a stopping time for the block-grid filtration because
\[
  \{J\le j\}
  =\bigcup_{r=1}^j\{\sigma=k-1+rN_\iota\}
  \in\F_{k-1+jN_\iota}.
\]
Moreover, \(\{\sigma\ge k\}\in\F_{k-1}\).  For each integer \(K\), conditional optional sampling and nonnegativity therefore give
\[
  \E_{P^\infty}\!\left[
    \one\{\sigma\ge k\}\Pi_J^{k,\iota}\one\{J\le K\}
    \mid\F_{k-1}
  \right]
  \le \one\{\sigma\ge k\}.
\]
Letting \(K\to\infty\) and applying conditional monotone convergence yields
\begin{equation}\label{eq:restart-bound}
  \E_{P^\infty}
  \left[
    \one\{\sigma\ge k\}\Xi_\sigma^{k,\iota}
    \mid\F_{k-1}
  \right]
  \le\one\{\sigma\ge k\}.
\end{equation}

Define the possibly extended-valued process
\begin{equation}\label{eq:sr-detector}
  \Detector_t:=\sum_{k=1}^t\sum_{\iota\in\mathcal I}
  w_\iota\Xi_t^{k,\iota}.
\end{equation}
For every bounded stopping time \(\sigma\), Tonelli's theorem and \eqref{eq:restart-bound} give
\[
  \E_{P^\infty}\Detector_\sigma
  \le\sum_{k\ge1}\Prob_{P^\infty}(\sigma\ge k)
  =\E_{P^\infty}\sigma.
\]
Lemma~\ref{lem:bounded-stopping-edetector} shows that \(\Detector\) is an e-detector, finite at every deterministic time almost surely under each null.

\emph{Calibration.}
Put \(\rho=(1+\eta)^{-1}\) and set
\[
  \Gamma_t:=\sum_{s=1}^tU_s.
\]
The process \((\Gamma_t-t)_{t\ge0}\) is a martingale for \((\F_t)\).  Hence bounded optional sampling gives \(\E\Gamma_\sigma=\E\sigma\) for every bounded \((\F_t)\)-stopping time \(\sigma\), and Lemma~\ref{lem:bounded-stopping-edetector} shows that \(\Gamma\) is an e-detector.  For the weighted process
\begin{equation}\label{eq:calibrated-detector}
  \widetilde{\Detector}_t
  :=(1-\rho)\Detector_t+\rho\Gamma_t,
\end{equation}
the defining inequality holds at every bounded stopping time by linearity; another application of Lemma~\ref{lem:bounded-stopping-edetector} therefore shows that \(\widetilde{\Detector}\) is an e-detector.  Proposition~\ref{prop:edetector-universality} shows that its full threshold family \((\widetilde T_\gamma)_{\gamma\ge1}\) is optionally ARL-valid.  If \(H_x=\inf\{t\ge1:\Gamma_t\ge x\}\), then \(H_x-1\) is the number of arrivals of a unit-rate Poisson process by time \(x\), so \(\E H_x=x+1\).  Since \(\widetilde{\Detector}_t\ge\rho\Gamma_t\),
\(\widetilde T_\gamma\le H_{\gamma/\rho}\).  Together with \eqref{eq:arl-lower}, this proves \eqref{eq:calibrated-arl}.

\emph{Detection delay.}
Fix \(Q\in\Qcal\), and choose \(m\) with
\(\mu=\E_{Q^{n_m}}h_m>0\).  Select a rational \(\lambda\in\mathcal K_m\) such that
\[
  \lambda H_m\le\frac12,
  \qquad
  \lambda H_m^2\le\frac\mu2,
\]
and let \(\iota=(m,\lambda)\).  Then
\begin{equation}\label{eq:positive-log-drift}
  g:=\E_{Q^{N_\iota}}\log G_\iota
  \ge\lambda\mu-\lambda^2\E_{Q^{n_m}}h_m^2
  \ge\frac{\lambda\mu}{2}>0.
\end{equation}
If the change occurs at time \(\nu\), the component indexed by \((k,\iota)=(\nu,\iota)\) uses only post-change data.  Put
\[
  Y_j:=\sum_{q=1}^j
  \log G_\iota\bigl(
    X_{\nu+(q-1)N_\iota},\ldots,
    X_{\nu+qN_\iota-1}
  \bigr),
\]
\[
  \sigma_b:=\inf\{j\ge1:Y_j\ge b\},
  \qquad
  b_\gamma:=\log\frac{\gamma}{(1-\rho)w_\iota}.
\]
Equations \eqref{eq:sparse-restart}, \eqref{eq:sr-detector}, and \eqref{eq:calibrated-detector} imply
\begin{equation}\label{eq:delay-by-hitting-time}
  (\widetilde T_\gamma-\nu+1)^+
  \le N_\iota\sigma_{b_\gamma}.
\end{equation}
Conditional on \(\F_{\nu-1}\), the increments of \(Y_j\) are i.i.d., bounded, independent of the pre-change history, and have mean \(g\).  Put \(b_+=\max(0,\sup\log G_\iota)\).  At every bounded stop \(\sigma_b\wedge K\), one has \(Y_{\sigma_b\wedge K}\le b+b_+\).  The conditional bounded-stop Wald identity, followed by monotone convergence, therefore gives
\begin{equation}\label{eq:hitting-time-bound}
  \E_{\nu}^{P,Q}[\sigma_b\mid\F_{\nu-1}]
  \le\frac{b+b_+}{g}.
\end{equation}
The right-hand side is deterministic and independent of \(P\), \(\nu\), and the realized pre-change history.  Therefore the same bound remains valid after taking the essential supremum and both outer suprema in \(\mathcal J_Q\).  Combining it with \eqref{eq:delay-by-hitting-time} proves \eqref{eq:log-delay}; relabel \(\widetilde{\Detector}\) as \(\Detector\).

\emph{Converse.}
The implications (ii)\(\Rightarrow\)(iii)\(\Rightarrow\)(iv) are immediate.  Assume (iv), and let \((T_\gamma)_{\gamma\ge1}\) be the fixed-horizon ARL-valid family appearing there.  For an integer \(\gamma\ge3\), define
\[
  \psi_\gamma:=\frac{(\gamma-T_\gamma)^+}{\gamma}
  =\frac1\gamma\sum_{t=1}^{\gamma-1}\one\{T_\gamma\le t\}.
\]
For each \(t<\gamma\), the event \(\{T_\gamma\le t\}\) depends only on the first \(t\) observations and the auxiliary variables revealed through time \(t\).  Thus \(\psi_\gamma\) depends only on \((X_1,\ldots,X_{\gamma-1})\) and the corresponding finite randomization.  Integrating out that randomization gives the measurable test
\begin{equation}\label{eq:detector-to-test}
  \varphi_\gamma(x_1,\ldots,x_{\gamma-1})
  :=\E_U[\psi_\gamma(x_1,\ldots,x_{\gamma-1},U)].
\end{equation}
For \(1\le t<\gamma\), fixed-horizon validity gives
\begin{equation}\label{eq:early-alarm-bound}
  \gamma\Prob_{P^\infty}(T_\gamma\le t)\le t.
\end{equation}
Consequently,
\begin{equation}\label{eq:test-null-size}
  \sup_{P\in\Pcal}\E_{P^\infty}\varphi_\gamma
  \le\frac1\gamma\sum_{t=1}^{\gamma-1}\frac{t}{\gamma}
  =\frac{\gamma-1}{2\gamma}<\frac12.
\end{equation}
For a change at the first observation,
\(\E_{Q^\infty}T_\gamma\le\mathcal J_Q(T_\gamma)\).  Whenever \eqref{eq:one-scale-delay} holds,
\begin{equation}\label{eq:test-alt-power}
  \E_{Q^\infty}\varphi_\gamma
  \ge1-\frac{\E_{Q^\infty}T_\gamma}{\gamma}
  >\frac12.
\end{equation}
Thus \((\gamma-1,\varphi_\gamma)_{\gamma\ge3}\) satisfies condition~\textup{(v)} of Theorem~\ref{thm:main}.  The common-e-detector conclusion follows from the construction in the implication (i)\(\Rightarrow\)(ii), and the final assertion follows from \eqref{eq:calibrated-arl}, \eqref{eq:log-delay}, and \(\log\gamma=o(\gamma)\).
\end{proof}

\subsection{Sharp pointwise detection rates}
\label{subsec:quantitative-detection}

The growth rate \(\Gamma(Q)\) also determines the sharp coefficient of
\(\log\gamma\) in the robust Lorden delay.  The converse uses only the
fixed-horizon false-alarm inequality \eqref{eq:fixed-horizon-alarm}.
The attaining construction is the threshold family of one e-detector
and therefore also satisfies optional-horizon validity.

\begin{proposition}[A finite-scale information lower bound]
\label{prop:detection-information-bound}
Fix \(\gamma\ge1\), and let \(T\) be an alarm satisfying
\[
  \sup_{P\in\Pcal}\Prob_{P^\infty}(T\le n)\le n/\gamma
  \qquad(n\ge1).
\]
For \(Q\) with \(\mu:=\E_{Q^\infty}T<\infty\) and
\(\Gamma(Q)<\infty\), one has
\begin{equation}\label{eq:detection-finite-information}
  \Gamma(Q)\mu+H_{\mathrm{geom}}(\mu)\ge\log\gamma,
  \qquad
  H_{\mathrm{geom}}(u):=u\log u-(u-1)\log(u-1),
  \quad u\ge1,
\end{equation}
where \(0\log0=0\).  In particular,
\begin{equation}\label{eq:detection-simple-information}
  \Gamma(Q)\mu+\log\mu+1\ge\log\gamma.
\end{equation}
Consequently, every family satisfying \eqref{eq:fixed-horizon-alarm}
obeys
\begin{equation}\label{eq:sharp-detection-lower}
  \liminf_{\gamma\to\infty}
  \frac{\mathcal J_Q(T_\gamma)}{\log\gamma}
  \ge \frac1{\Gamma(Q)}
  \qquad\text{if }0<\Gamma(Q)<\infty.
\end{equation}
The alarms may use independent auxiliary randomization.
\end{proposition}

\begin{proof}
For every \(P\in\Pcal\) and \(\lambda>0\), summation by parts gives
\begin{align}
  \E_{P^\infty}e^{-\lambda T}
  &=(1-e^{-\lambda})\sum_{t=1}^{\infty}
       e^{-\lambda t}\Prob_{P^\infty}(T\le t)
    \notag\\
  &\le \frac{1}{\gamma(e^\lambda-1)}.
  \label{eq:alarm-laplace-bound}
\end{align}
Here \(e^{-\lambda\infty}=0\).  Restart \(T\) on fresh observations
after each alarm, using fresh independent auxiliary randomness on
each run.  Under any fixed i.i.d.\ law, the successive run lengths
\(T^{(1)},T^{(2)},\ldots\) are independent copies of \(T\), with an
infinite run length preventing all subsequent alarms.  Put
\[
  S_k:=\sum_{i=1}^k T^{(i)},\qquad
  A_{k,n}:=\{S_k\le n\}.
\]
The event \(A_{k,n}\) is determined by the first \(n\) observations
and the independent randomization: it records whether \(k\) fresh
runs have finished by time \(n\).  By \eqref{eq:alarm-laplace-bound},
\begin{equation}\label{eq:renewal-alarm-size}
  a_{k,n}:=\sup_{P\in\Pcal}\Prob_{P^\infty}(A_{k,n})
  \le
  \frac{e^{\lambda n}}{[\gamma(e^\lambda-1)]^k}.
\end{equation}

Fix \(\delta>0\), and let \(n_k=\lceil k(\mu+\delta)\rceil\).
The law of large numbers under \(Q\) implies
\(q_k:=\Prob_{Q^\infty}(A_{k,n_k})\to1\).
For an event \(A\) of worst-case null probability \(a\in(0,1]\),
\[
  e_A:=\frac{1+a^{-1}\one_A}{2}
\]
is bounded, bounded away from zero, and has null expectation at
most one.  Moreover,
\[
  \E_Q\log e_A\ge Q(A)\log(1/a)-\log2.
\]
If \(a=0<Q(A)\), the valid bounded bets
\(1+L\one_A\), \(L\ge1\), instead imply \(G_n(Q)=\infty\).
When the event uses auxiliary randomization \(U\), averaging
\(e_A\) over \(U\) gives a bounded observation-only e-variable;
conditional Jensen's inequality cannot decrease its expected
logarithm under \(Q\).

Apply these observations to \eqref{eq:renewal-alarm-size}.
Since \(G_{n_k}(Q)/n_k\le\Gamma(Q)\), whenever the expression on
the right below is positive,
\[
  \Gamma(Q)\ge
  \liminf_{k\to\infty}\frac{G_{n_k}(Q)}{n_k}
  \ge
  \frac{\log[\gamma(e^\lambda-1)]}{\mu+\delta}-\lambda.
\]
If it is nonpositive the same inequality follows from
\(\Gamma(Q)\ge0\).  Letting \(\delta\downarrow0\) yields
\begin{equation}\label{eq:renewal-information-bound}
  \Gamma(Q)\mu
  \ge \log\gamma+\log(e^\lambda-1)-\lambda\mu
  \qquad(\lambda>0).
\end{equation}
For \(\mu>1\), the right side is maximized at
\(\lambda=\log(\mu/(\mu-1))\); for \(\mu=1\), take
\(\lambda\to\infty\).  This proves
\eqref{eq:detection-finite-information}.
The inequality
\[
  H_{\mathrm{geom}}(u)
  =\log u+(u-1)\log\!\left(1+\frac1{u-1}\right)
  \le\log u+1
  \qquad(u>1)
\]
and the case \(u=1\) give \eqref{eq:detection-simple-information}.

Finally, \(\E_{Q^\infty}T_\gamma\le\mathcal J_Q(T_\gamma)\)
by taking \(\nu=1\).
The left side of \eqref{eq:detection-simple-information} is
increasing in \(\mu\), so the inequality remains true with
\(\mathcal J_Q(T_\gamma)\) in place of \(\mu\).
Along any sequence on which
\(\mathcal J_Q(T_\gamma)/\log\gamma\) is bounded,
\(\log\mathcal J_Q(T_\gamma)=o(\log\gamma)\).
Dividing by \(\log\gamma\) proves
\eqref{eq:sharp-detection-lower}; sequences on which this ratio
diverges require no further argument.
\end{proof}

\begin{theorem}[Countable recovery and optimal change detection]
\label{thm:quantitative-change-detection}
Suppose \(0<\Gamma(Q)\le\infty\) for every \(Q\in\Qcal\).
The following statements are equivalent.
\begin{enumerate}[label=(\roman*)]
\item The countable rate-recovery condition
\eqref{eq:countable-rate-recovery} holds, equivalently the
conditions of Theorem~\ref{thm:quantitative-testing} hold.

\item There exists a fixed-horizon ARL-valid family
\((T_\gamma)_{\gamma\ge1}\) such that
\[
  \limsup_{\gamma\to\infty}
  \frac{\mathcal J_Q(T_\gamma)}{\log\gamma}
  \le\frac1{\Gamma(Q)}
  \qquad\text{for every }Q\in\Qcal,
\]
with \(1/\infty=0\).

\item For every \(\eta>0\), there exists one e-detector whose
threshold family satisfies \eqref{eq:calibrated-arl} and
\begin{equation}\label{eq:optimal-detection-delay}
  \lim_{\gamma\to\infty}
  \frac{\mathcal J_Q(T_\gamma)}{\log\gamma}
  =\frac1{\Gamma(Q)}
  \qquad\text{for every }Q\in\Qcal.
\end{equation}
\end{enumerate}
In particular, the family in \textup{(iii)} is optionally ARL-valid.
The limits are pointwise in \(Q\); no uniformity over \(\Qcal\)
is asserted.
\end{theorem}

\begin{proof}
Assume \textup{(i)}, and let
\((N_\iota,G_\iota)_{\iota\in\mathcal I}\) be a countable
family of block e-variables, bounded and bounded away from zero, realizing
\eqref{eq:countable-rate-recovery}.
Use these factors, with fixed positive weights \(w_\iota\)
summing to one, in the restarted construction
\eqref{eq:sparse-restart}--\eqref{eq:sr-detector}.
Its validity proof uses only the block expectation bound
\eqref{eq:block-e-variable}, so it applies to this family.
The calibration \eqref{eq:calibrated-detector}, with
\(\rho=(1+\eta)^{-1}\), gives one e-detector satisfying
\eqref{eq:calibrated-arl}.

Fix \(Q\), and choose any component with
\[
  g_\iota(Q):=\E_{Q^{N_\iota}}\log G_\iota>0,
  \qquad
  b_\iota:=\sup(\log G_\iota)^+<\infty.
\]
The component restarted at the true changepoint \(\nu\) gives
\eqref{eq:delay-by-hitting-time}, with boundary
\[
  b_\gamma=\log\frac{\gamma}{(1-\rho)w_\iota}.
\]
For its log-wealth random walk \(Y_j\), let
\(\sigma_b=\inf\{j\ge1:Y_j\ge b\}\).
For every deterministic integer \(K\),
\(Y_{\sigma_b\wedge K}\le b+b_\iota\) pathwise.
The bounded-stop Wald identity, conditionally on
\(\F_{\nu-1}\), therefore gives
\[
  g_\iota(Q)
  \E_{\nu}^{P,Q}[\sigma_b\wedge K\mid\F_{\nu-1}]
  =
  \E_{\nu}^{P,Q}[Y_{\sigma_b\wedge K}\mid\F_{\nu-1}]
  \le b+b_\iota.
\]
Monotone convergence, followed by the essential and ordinary
suprema in the definition of \(\mathcal J_Q\), yields
\begin{equation}\label{eq:sharp-component-delay}
  \mathcal J_Q(T_\gamma)
  \le
  \frac{N_\iota}{g_\iota(Q)}
  \left[
    \log\gamma+
    \log\frac1{(1-\rho)w_\iota}+b_\iota
  \right].
\end{equation}
For each fixed component its mixture penalty and overshoot bound
are constant in \(\gamma\).  Taking the supremum of
\(g_\iota(Q)/N_\iota\) over the recovery family proves
the upper bound in \eqref{eq:optimal-detection-delay}.
Proposition~\ref{prop:detection-information-bound} supplies
the matching lower bound when \(\Gamma(Q)<\infty\).
When \(\Gamma(Q)=\infty\), choosing components of arbitrarily
large normalized drift proves
\(\mathcal J_Q(T_\gamma)=o(\log\gamma)\).
This establishes \textup{(iii)}, which implies \textup{(ii)}.

Assume \textup{(ii)}.
For every integer \(m\ge2\), restart \(T_m\) successively as
in the proof of Proposition~\ref{prop:detection-information-bound},
and include all events
\[
  A_{m,k,n}:=\{T_m^{(1)}+\cdots+T_m^{(k)}\le n\},
  \qquad k,n\ge1.
\]
There are countably many such events.
For each event of positive worst-case null probability,
include the bounded positive factor \(e_A\) used in that proof.
For an event of zero worst-case null probability, include
\(1+L\one_A\) for every positive integer \(L\).
Average over the auxiliary randomization, if present.
Together with the constant factor one, these form a countable
family \(\mathcal F\) of observation-only block e-variables, each
bounded and bounded away from zero.
Let
\[
  R_{\mathcal F}(Q):=
  \sup_{(n,e)\in\mathcal F}\frac1n\E_{Q^n}\log e
  \le\Gamma(Q).
\]

Fix \(Q\).  For all sufficiently large \(m\),
\(\mu_m:=\E_{Q^\infty}T_m\le\mathcal J_Q(T_m)<\infty\).
The proof of \eqref{eq:renewal-information-bound}, using only
the factors just included in \(\mathcal F\), gives
\[
  R_{\mathcal F}(Q)
  \ge
  \frac{\log m-\log\mu_m-1}{\mu_m}.
\]
If \(0<\Gamma(Q)<\infty\), condition \textup{(ii)} gives
\(\mu_m\le(1/\Gamma(Q)+o(1))\log m\).
The last display therefore implies
\(R_{\mathcal F}(Q)\ge\Gamma(Q)\).
If \(\Gamma(Q)=\infty\), then
\(\mu_m=o(\log m)\), and the displayed lower bound diverges.
Thus \(R_{\mathcal F}(Q)=\Gamma(Q)\) for every \(Q\), proving
\textup{(i)}.
\end{proof}

\begin{corollary}[Detection at attainable rate profiles]
\label{cor:detection-rate-profiles}
Let \(r:\Qcal\to(0,\infty]\) satisfy \(r(Q)\le\Gamma(Q)\).
The following statements are equivalent.
\begin{enumerate}[label=(\roman*)]
\item There is a countable family of block e-variables
\((n_j,e_j)_{j\ge1}\), each bounded and bounded away from zero,
such that
\[
  \sup_j\frac1{n_j}\E_{Q^{n_j}}\log e_j\ge r(Q)
  \qquad(Q\in\Qcal).
\]
\item There is a fixed-horizon ARL-valid family
\((T_\gamma)_{\gamma\ge1}\) satisfying
\begin{equation}\label{eq:detection-rate-profile}
  \limsup_{\gamma\to\infty}
  \frac{\mathcal J_Q(T_\gamma)}{\log\gamma}
  \le \frac1{r(Q)}
  \qquad(Q\in\Qcal),
\end{equation}
where \(1/\infty=0\).
\item For every \(\eta>0\), one e-detector has a threshold family
satisfying both \eqref{eq:calibrated-arl} and
\eqref{eq:detection-rate-profile}.
\end{enumerate}
\end{corollary}

\begin{proof}
Under \textup{(i)}, the restarted and calibrated mixture in
Theorem~\ref{thm:quantitative-change-detection} has, by
\eqref{eq:sharp-component-delay},
\[
  \limsup_{\gamma\to\infty}
  \frac{\mathcal J_Q(T_\gamma)}{\log\gamma}
  \le
  \left(\sup_j n_j^{-1}\E_{Q^{n_j}}\log e_j\right)^{-1}
  \le r(Q)^{-1}.
\]
This proves \textup{(iii)}, which implies \textup{(ii)}.
Conversely, apply the countable renewal-event construction from
that theorem to the integer-scale alarms in \textup{(ii)}.
It produces one countable family \(\mathcal F\) of bounded block
e-variables, bounded away from zero, whose envelope satisfies
\[
  R_{\mathcal F}(Q)
  \ge
  \frac{\log m-\log\mu_m-1}{\mu_m},
  \qquad
  \mu_m:=\E_{Q^\infty}T_m\le\mathcal J_Q(T_m),
\]
for all sufficiently large integers \(m\).
If \(r(Q)<\infty\), then
\(\mu_m\le(r(Q)^{-1}+o(1))\log m\), and the displayed bound
implies \(R_{\mathcal F}(Q)\ge r(Q)\).
If \(r(Q)=\infty\), then \(\mu_m=o(\log m)\), and the bound
diverges.  Thus \(\mathcal F\) satisfies \textup{(i)}.
\end{proof}

Together with Corollary~\ref{cor:testing-rate-profiles}, this transfers every attainable rate profile to expected testing costs and robust detection delays.  In Example~\ref{ex:lalley-nobel-rate-gap}, the best common worst-case pointwise logarithmic coefficient is \(1/c\) for both tasks, although each alternative's oracle coefficient is \(1/(c+d)\).  Here ``worst-case pointwise'' takes the asymptotic limit superior separately at each fixed alternative before taking the supremum over alternatives.

\begin{corollary}[Sufficient conditions for sharp detection benchmarks]\label{cor:sharp-detection-sufficient}
If \(\Qcal\) is total-variation separable and \(\Gamma(Q)>0\) throughout \(\Qcal\), one calibrated e-detector attains \eqref{eq:optimal-detection-delay}.  This includes alternatives dominated by one sigma-finite measure on a Polish Borel space, with no restriction on the null.  If instead \(\Phi\) is weakly lower semicontinuous at every alternative and \(\Phi(Q)>0\), then the same statement holds with \(\Phi(Q)\) in place of \(\Gamma(Q)\).  In particular, it holds against every alternative outside a weakly compact null on a Polish space.
\end{corollary}

\begin{proof}
Apply Proposition~\ref{prop:quantitative-sufficient} and Theorem~\ref{thm:quantitative-change-detection}.  For a weakly compact null, \(\Phi(Q)>0\) outside the null: otherwise Pinsker's inequality would produce null laws converging weakly to \(Q\), contradicting weak closedness.
\end{proof}

\begin{remark}[Zero and infinite information rates]
\label{rem:detection-zero-infinite-rates}
If \(\Gamma(Q)=0\), every bounded finite-block test has
\(Q\)-power at most its worst-case null size; a strict gap would
give a bounded block bet with positive expected logarithm.
Consequently, an alarm satisfying
\eqref{eq:fixed-horizon-alarm} also obeys
\[
  \Prob_{Q^\infty}(T_\gamma\le n)\le n/\gamma.
\]
This remains true with auxiliary randomization by integrating it
out.  Hence
\[
  \mathcal J_Q(T_\gamma)
  \ge\E_{Q^\infty}T_\gamma
  \ge\sum_{n=0}^{\infty}(1-n/\gamma)^+
  \ge\gamma/2.
\]
In particular, such a \(Q\) cannot have sublinear delay in
\(\gamma\).  If optional-horizon validity is assumed, the stronger
bound \(\E_{Q^\infty}T_\gamma\ge\gamma\) holds: for bounded
\(\sigma\), the bounded score
\(\gamma\one\{T_\gamma\le\sigma\}-\sigma\) has nonpositive
expectation under every null and therefore also under \(Q\).
Apply this observation to \(\sigma=T_\gamma\wedge n\) and
let \(n\to\infty\).

At the other endpoint, \(\Gamma(Q)=\infty\) gives
\(o(\log\gamma)\) delay in Theorem~\ref{thm:quantitative-change-detection}.
Threshold-uniform bounded delay is a stronger conclusion, obtained
from the zero-null-probability events in
Corollary~\ref{cor:sure-certificates}.
An almost-sure log-growth limit for an arbitrary e-process alone
does not ensure the expected-delay bound
\eqref{eq:sharp-component-delay}; the bounded block factors and
fresh restarts supply that bound.
\end{remark}

\begin{remark}[The Lorden lower bound and ordinary ARL]\label{rem:lorden-lower-bound}
Under ordinary ARL control, the classical i.i.d. Lorden lower bound applies to each fixed pair \(P,Q\) with \(0<\KL(Q\|P)<\infty\), under its usual change-of-measure hypotheses \cite{lorden1971,tartakovsky-et-al-2014,ram-ramdas-bounded-means-2026}.  Thus, when \(0<\Phi(Q)<\infty\), it gives
\[
  \liminf_{\gamma\to\infty}\frac{\mathcal J_Q(T_\gamma)}{\log\gamma}
  \ge\frac1{\Phi(Q)}.
\]
One applies the simple-pair result to a fixed null law arbitrarily close to the KL infimum; no uniform limit over \(P\) is needed.  Under fixed-time false-alarm control, Proposition~\ref{prop:detection-information-bound} strengthens this benchmark to \(1/\Gamma(Q)\) without domination assumptions.  When \(\Gamma=\Phi\) and countable recovery holds, our common detector is consequently also asymptotically optimal among ordinary-ARL procedures.  In general, ordinary ARL does not imply the intrinsic-rate bound or the qualitative testability criterion; Example~\ref{ex:ordinary-arl-gap} and Remark~\ref{rem:ordinary-arl-positive-rate} show the failures.
\end{remark}

\begin{remark}[External randomization does not enlarge the existence class]\label{rem:no-randomization}
Theorem~\ref{thm:main} already shows that randomized finite-block certificates do not enlarge the class of sequentially testable pairs, because they are equivalent to deterministic events and to a nonrandomized stopping rule.  The same is true for change detection.  In the preceding proof, the exponential clock \(\Gamma_t\) may be replaced by the deterministic e-detector \(\Gamma_t=t\).  Then \(\widetilde T_\gamma\le\lceil\gamma/\rho\rceil\) pathwise, with the same logarithmic evidence bound.

The exponential clock is nevertheless a more natural calibration device because it creates no deterministic terminal time.  Conditional on the current clock value \(\Gamma_t=c<\gamma/\rho\), its expected additional number of increments needed to cross is \(\gamma/\rho-c+1\); the remaining false-alarm budget is random and state dependent.  By contrast, the deterministic clock forces an alarm at a known deadline and its residual time collapses pathwise as that deadline approaches.  Thus randomization is unnecessary for existence, but the exponential clock better reflects residual-ARL reasoning.
\end{remark}

\begin{example}[One-sided mean changes]\label{ex:mean-change}
Let \(\X=[0,1]\), fix \(m\in(0,1)\), and take
\[
  \Pcal=\{P:\E_PX\le m\},
  \qquad
  \Qcal=\{Q:\E_QX>m\}.
\]
The one-observation score \(h(x)=x-m\) satisfies \(\sup_{P\in\Pcal}\E_Ph\le0\), while every \(Q\in\Qcal\) has
\[
  \delta_Q:=\E_QX-m=\E_Qh>0.
\]
Thus \Cref{thm:main,thm:change-detection} apply.  The drift calculation also displays the delay scale.  For \(0<\lambda\le1/2\), put \(G_\lambda=1+\lambda h\).  Since \(|h|\le1\) and \(\log(1+u)\ge u-u^2\) for \(|u|\le1/2\),
\[
  \E_Q\log G_\lambda\ge\lambda\delta_Q-\lambda^2.
\]
Choosing a rational \(\lambda\in[\delta_Q/4,\delta_Q/2]\) gives
\[
  \E_Q\log G_\lambda\ge\frac{3}{16}\delta_Q^2.
\]
Consequently, for each fixed calibration parameter \(\eta>0\), the construction in the proof of Theorem~\ref{thm:change-detection} yields a finite constant \(C_{Q,\eta}\) and a universal \(C\) such that
\[
  \mathcal J_Q(T_\gamma)
  \le C_{Q,\eta}+C\delta_Q^{-2}\log\gamma,
  \qquad \gamma\ge2.
\]
The constant \(C_{Q,\eta}\) includes the calibration and the fixed mixture-weight penalty for the selected rational \(\lambda\).

For \(0<\Delta\le1-m\), consider the uniformly separated class
\[
  \Qcal_\Delta:=\{Q:\E_QX\ge m+\Delta\}.
\]
The single choice \(\lambda=\Delta/4\) has log drift at least \(3\Delta^2/16\).  For this restricted post-change problem, use the single-component detector generated by this factor, with mixture weight one, before adding the calibration clock; equivalently, use the construction in Corollary~\ref{cor:uniform-change} below.  The log-increment range is at most a universal constant times \(\Delta\), so \eqref{eq:hitting-time-bound} yields, for each fixed calibration parameter \(\eta>0\),
\[
  \sup_{Q\in\Qcal_\Delta}\mathcal J_Q(T_\gamma)
  \le C_\eta\Delta^{-2}(1+\log\gamma),
\]
where \(C_\eta\) is independent of \(\Delta\).  If one instead embeds the factor in a universal countable mixture, an additional \(\Delta^{-2}\log(1/w_\Delta)\) term records the weight assigned to the chosen rational betting fraction.
No uniform conclusion is possible over all of \(\Qcal\).  Taking \(P_0=\delta_m\) and \(Q_\varepsilon=(1-\varepsilon)\delta_m+\varepsilon\delta_1\), one has \(Q_\varepsilon\in\Qcal\) but, for each fixed \(n\),
\[
  d_{\TV}(Q_\varepsilon^n,P_0^n)
  =1-(1-\varepsilon)^n\longrightarrow0.
\]
Thus every fixed positive mean gap is detected with an \(O_Q(\log\gamma)\) bound, while the constants necessarily deteriorate as the post-change law approaches the null boundary.  Sharp first-order lower and upper bounds for bounded-mean changes are studied in \cite{ram-ramdas-bounded-means-2026}.
\end{example}

\begin{corollary}[Uniform post-change class]\label{cor:uniform-change}
The following are equivalent; condition~\textup{(i)} is also equivalent to the uniform testing conditions of Proposition~\ref{prop:uniform-testing}.
\begin{enumerate}[label=(\roman*)]
\item For some \(n\ge1\), there exists a randomized test \(\varphi:\X^n\to[0,1]\) such that
\[
  \inf_{Q\in\Qcal}\E_{Q^n}\varphi
  >\sup_{P\in\Pcal}\E_{P^n}\varphi.
\]

\item For every \(\eta>0\), there is an optionally ARL-valid family satisfying \eqref{eq:calibrated-arl} and
\[
  \sup_{Q\in\Qcal}\mathcal J_Q(T_\gamma)=O_\eta(\log\gamma).
\]

\item There is a fixed-horizon ARL-valid family for which
\[
  \sup_{Q\in\Qcal}
  \frac{\mathcal J_Q(T_\gamma)}{\gamma}
  \longrightarrow0
  \qquad\text{as }\gamma\to\infty.
\]
\end{enumerate}
In condition~\textup{(ii)}, the family may be chosen as the threshold family of one common e-detector.  In this case, one common \(\gamma_0\) satisfies
\[
  \mathcal J_Q(T_\gamma)
  <\operatorname{ARL}_{\Pcal}(T_\gamma)
  \qquad
  \text{for every }Q\in\Qcal
  \text{ and every }\gamma\ge\gamma_0.
\]
By \eqref{eq:block-tv-identity}, condition~\textup{(i)} is equivalently
\[
  d_{\TV}\bigl(\mathscr C_n,\mathscr D_n(\Qcal)\bigr)>0
\]
for some \(n\ge1\).
\end{corollary}

\begin{proof}
By the proof of Proposition~\ref{prop:uniform-testing}, condition~\textup{(i)} supplies a single bounded block factor with null expectation at most one, uniformly positive logarithmic drift over \(\Qcal\), and a bounded log-increment range.  Using this factor as the single evidence component in the construction of Theorem~\ref{thm:change-detection} yields condition~\textup{(ii)} with constants uniform in \(Q\), and the resulting family is generated by one e-detector.  The implication (ii)\(\Rightarrow\)(iii) is immediate.  If (iii) holds, choose one sufficiently large integer \(\gamma\) such that \(\sup_Q\mathcal J_Q(T_\gamma)<\gamma/2\).  The triangular test in the converse proof of Theorem~\ref{thm:change-detection}, defined in \eqref{eq:detector-to-test}, then has null size below \(1/2\) and power above \(1/2\) uniformly over \(\Qcal\).  The resulting test is condition~(i).  The last statement follows from the optional-horizon ARL lower bound \eqref{eq:arl-lower}.
\end{proof}

\begin{remark}[A uniform--pointwise--impossible trichotomy]\label{rem:trichotomy}
A finite cover by uniformly separated block classes is equivalent to one uniform finite-block certificate.  Indeed, amplify each of the finitely many certificates as in Remark~\ref{rem:amplification} on disjoint blocks, and reject if any amplified event occurs; with sufficiently small component errors, the union has a positive uniform power--size gap.  The main results therefore divide the problem into three regimes:
\begin{center}
\small
\renewcommand{\arraystretch}{1.25}
\begin{tabular}{@{}>{\raggedright\arraybackslash}p{0.17\textwidth}>{\raggedright\arraybackslash}p{0.40\textwidth}>{\raggedright\arraybackslash}p{0.34\textwidth}@{}}
\textbf{Regime} & \textbf{Finite-block structure} & \textbf{Sequential consequence} \\
\hline
Uniform & One uniform certificate, equivalently a finite positive-margin cover & Uniform testing by Proposition~\ref{prop:uniform-testing} and optionally ARL-valid detection with \(\sup_Q\mathcal J_Q(T_\gamma)=O(\log\gamma)\); Example~\ref{ex:mean-change} with \(\Qcal_\Delta\) \\
Pointwise only & A countable cover exists but no uniformly separated finite cover does & Pointwise \(O_Q(\log(1/\alpha))\) testing and \(O_Q(\log\gamma)\) detection-delay upper bounds; the full mean-shift class in Example~\ref{ex:mean-change} \\
Impossible & No countable cover by finite-block certificates & No power-one test and no fixed-horizon ARL-valid family with pointwise sub-ARL delay; Examples~\ref{ex:countable-cocountable} and~\ref{ex:lalley-nobel-rate-gap}
\end{tabular}
\end{center}
\end{remark}

\begin{remark}[Why the optional-horizon ARL formulation is nonvacuous]
Optional-horizon validity implies both robust ARL at least \(\gamma\) and the finite-horizon bound \(\Prob_{P^\infty}(T_\gamma\le t)\le t/\gamma\), so it rules out obtaining a large mean merely by concentrating excessive false-alarm probability near the beginning and compensating with a long right tail.  It does not, by itself, force the null-law expected run lengths to be finite.  For any \(\eta>0\), the calibrated common e-detector \eqref{eq:calibrated-detector} forces every null-law expected run length into the finite interval \([\gamma,(1+\eta)\gamma+1]\) while preserving an \(O_Q(\log\gamma)\) detection-delay bound.  The independent exponential clock has unbounded support and avoids a predetermined alarm deadline; Remark~\ref{rem:no-randomization} records the deterministic alternative and explains the residual-time distinction.  In either form, every calibrated null-law expected run length and the threshold remain on the same scale.
\end{remark}

\begin{example}[A bare comparison with robust ARL can be vacuous]\label{ex:infinite-arl}
Return to Example~\ref{ex:two-observations} and define
\[
  T:=\inf\{t\ge2:X_t\ne X_{t-1}\}.
\]
Under every no-change law \(\delta_x^\infty\), one has \(T=\infty\) almost surely, so the robust ARL is infinite.  Under any non-point-mass post-change law \(Q\), let
\[
  \beta_Q:=Q^2(X_1\ne X_2)>0.
\]
After any changepoint, inspect disjoint pairs of post-change observations.  Their inequality indicators are independent Bernoulli variables with success probability \(\beta_Q\).  Ignoring any earlier alarm therefore gives the robust bound
\[
  \mathcal J_Q(T)\le\frac{2}{\beta_Q}<\infty.
\]
Hence \(\mathcal J_Q(T)<\operatorname{ARL}_{\Pcal}(T)\) holds for every \(Q\), but only because the right-hand side is infinite.  The calibrated optionally ARL-valid family in Theorem~\ref{thm:change-detection} rules out this vacuous mechanism.
\end{example}

The next result shows that the finite delay in Example~\ref{ex:infinite-arl} is not merely an artifact of an infinite, and hence vacuous, ARL: zero-null-probability certificates retain threshold-uniform delay after finite-ARL calibration.

\begin{corollary}[Sure certificates yield threshold-uniform delay]\label{cor:sure-certificates}
Suppose there is a countable family \((n_m,B_m)\) such that
\[
  \sup_{P\in\Pcal}P^{n_m}(B_m)=0
\]
for every \(m\), and every \(Q\in\Qcal\) has \(Q^{n_m}(B_m)>0\) for at least one \(m\).  Then, for every \(\eta>0\), the calibrated detector can be chosen so that \eqref{eq:calibrated-arl} holds and
\[
  \sup_{\gamma\ge2}\mathcal J_Q(T_\gamma)<\infty
  \qquad\text{for every fixed }Q\in\Qcal.
\]
\end{corollary}

\begin{proof}
For each \(m\) and integer \(\ell\ge1\), use the factor
\[
  G_{m,\ell}=1+\ell\one_{B_m}.
\]
It has expectation one under every null.  Construct one \emph{fixed} restarted detector containing all pairs \((m,\ell)\), with weights \(v_m c\ell^{-2}\), where \(v_m>0\), \(\sum_mv_m=1\), and \(c^{-1}=\sum_{\ell\ge1}\ell^{-2}\).  Fix \(Q\), choose \(m\) with \(p=Q^{n_m}(B_m)>0\), and fix \(\rho=(1+\eta)^{-1}\).  The detector itself does not depend on \(\gamma\); only the analysis selects, for each threshold, an integer \(\ell=\ell(\gamma)\) such that
\[
  (1-\rho)v_mc\ell^{-2}(1+\ell)^3\ge\gamma.
\]
Such an integer exists because \(\ell^{-2}(1+\ell)^3\to\infty\).  Three successful blocks are the first power for which this happens: \(\ell^{-2}(1+\ell)^j\) remains bounded for \(j\le2\).  The selected component therefore crosses after its third successful post-change block.  Conditional on any pre-change history and on a change at time \(\nu\), the successive post-change block events are i.i.d. Bernoulli with success probability \(p\).  The number of blocks required for the third success consequently has negative-binomial mean \(3/p\), independently of \(P\), \(\nu\), the history, \(\ell\), and \(\gamma\).  Taking the conditional essential supremum therefore gives
\[
  \mathcal J_Q(T_\gamma)\le \frac{3n_m}{p}
\]
for every \(\gamma\ge2\).  The calibration clock can only decrease the stopping time.
\end{proof}

\begin{remark}[Why the converse is not immediate]\label{rem:sure-converse}
The alarm-family-to-test argument in \eqref{eq:detector-to-test} does not resolve the converse.  Even if \(\sup_\gamma\mathcal J_Q(T_\gamma)<\infty\), it produces tests whose \(Q\)-power tends to one but whose worst-case null size is bounded only by
\[
  \frac{\gamma-1}{2\gamma}\longrightarrow\frac12,
\]
not by a quantity tending to zero.  Thus bounded detection delay does not automatically yield a zero-null-probability finite-block certificate; a different idea would be required.
\end{remark}

\subsection{Ordinary ARL alone}

\begin{example}[Ordinary ARL does not characterize testability]\label{ex:ordinary-arl-gap}
The fixed-time false-alarm requirement cannot be replaced by ordinary ARL, even if every null-law mean is finite and calibrated to within an additive constant.  Let
\(\X=[0,1]\times\{0,1\}^{\mathbb N}\), write an observation as \((U,Z)\), and let \(Q\) make \(U\) uniform and the coordinates of \(Z\) independent fair bits, independently of \(U\).  For \(k\ge1\), let \(P_k\) have the same law except that \(Z_k=1\) surely, and put \(\Pcal=\{P_k:k\ge1\}\).
Then \(\Gamma(Q)=0\), so even this single alternative has no finite-block certificate.  Nevertheless, there is a data-only alarm family with
\begin{equation}\label{eq:ordinary-arl-gap}
  \gamma\le\E_{P_k^\infty}T_\gamma<\gamma+2
  \quad(k\ge1,\ \gamma\ge1),
  \qquad
  \mathcal J_Q(T_\gamma)\le2.
\end{equation}
\end{example}

\begin{proof}
For fixed \(n\), the densities
\[
  H_{k,n}:=\frac{dP_k^n}{dQ^n}
  =2^n\one\{Z_{1,k}=\cdots=Z_{n,k}=1\}
\]
are independent under \(Q^n\), with mean one and variance \(2^n-1\).  Consequently
\[
  \E_{Q^n}\left|\frac1M\sum_{k=1}^M H_{k,n}-1\right|^2
  =\frac{2^n-1}{M}\longrightarrow0.
\]
Thus \(Q^n\) is in the total-variation closure of the convex hull of \(\{P_k^n:k\ge1\}\).  If \(f\ge0\) is any \(\Pcal^n\)-e-variable, apply this convergence to \(f\wedge K\), and then let \(K\uparrow\infty\), to obtain \(\E_{Q^n}f\le1\).  Jensen's inequality gives \(G_n(Q)=0\) for every \(n\), hence \(\Gamma(Q)=0\).  The same closure directly rules out a strict finite-block power--size gap.

Choose \(w_j>0\) with \(\sum_jw_j=1\), partition \([0,1]\) into Borel intervals \(I_j\) of lengths \(w_j\), and set \(J=j\) when \(U_1\in I_j\).  Define
\[
  \tau_j:=\inf\{t\ge1:Z_{t,j}=0\},
  \qquad
  T_\gamma:=\min\left\{\tau_J,
                 \left\lceil\frac\gamma{w_J}\right\rceil\right\}.
\]
This is a stopping time for the data filtration.  Under \(P_k^\infty\), the branch \(J=k\) has duration \(\lceil\gamma/w_k\rceil\), whereas on every branch \(J\ne k\) the uncapped duration is geometric with mean two.  Therefore
\[
  \gamma\le
  w_k\left\lceil\frac\gamma{w_k}\right\rceil
  \le\E_{P_k^\infty}T_\gamma
  \le w_k\left\lceil\frac\gamma{w_k}\right\rceil+2(1-w_k)
  <\gamma+2.
\]
After a change to \(Q\), conditional on any surviving pre-change history, a zero in the selected coordinate has probability \(1/2\) independently at each subsequent observation.  The remaining deterministic cap can only shorten the wait.  This gives conditional expected delay at most two.  At \(\nu=1\), the same argument holds because \(U_1\), hence \(J\), is independent of all the bits.  Taking the essential supremum over histories, changepoints, and pre-change laws proves \eqref{eq:ordinary-arl-gap}.
\end{proof}

Here \(P_k^\infty(T_\gamma=1)=(1-w_k)/2\), independently of \(\gamma\).  Ordinary ARL is supplied by the rare branch selecting the correct null coordinate; a genuine change makes even that branch terminate quickly.  Thus robust Lorden loss and finite, nearly equal null ARLs do not recover the missing fixed-time restriction.  By \cite[Theorem~5.8]{ramdas-edetector-universality-2026}, each fixed-scale alarm has a weak e-detector representation, but this does not give the optional-horizon inequality.

\begin{remark}[Failure also at a positive finite oracle rate]\label{rem:ordinary-arl-positive-rate}
Append an independent Bernoulli coordinate, with full-support laws \(B_0\) under every null and \(B_1\) under the alternative, and put \(a=\KL(B_1\|B_0)\in(0,\infty)\).  The resulting oracle rate is exactly \(a\): the Bernoulli likelihood ratio attains it, while the preceding mixtures converge, on each block, to \(Q^n\otimes B_0^n\), which bounds the expected log of every block e-variable by \(na\).  The alarms above ignore this coordinate and retain \eqref{eq:ordinary-arl-gap}.  Hence ordinary ARL also invalidates a \((\log\gamma)/\Gamma(Q)\) lower bound when \(0<\Gamma(Q)<\infty\).  Here the raw KL projection \(\Phi(Q)\) remains infinite, so the ordinary-ARL bound in Remark~\ref{rem:lorden-lower-bound} is vacuous.
\end{remark}

\section{Applications of the finite-block criterion}\label{sec:applications}

The following consequences require no additional sequential arguments: it suffices to exhibit the bounded scores in condition~\textup{(vi)} of Theorem~\ref{thm:main}.  They complement constructive betting procedures for two-sample, independence, and related nonparametric problems \cite{shekhar-ramdas-2024,podkopaev-et-al-2023,podkopaev-ramdas-2023,pandeva-et-al-2024,shaer-et-al-2023,henzi-law-2024}.

\begin{example}[Universal sequential goodness-of-fit]\label{ex:gof}
Fix \(P_0\in\M_1(\X)\) and suppose \(\B\) is countably generated.  Choose a countable generating algebra \(\{A_j:j\ge1\}\); by the \(\pi\)-\(\lambda\) theorem it is measure determining.  Standard Borel spaces are an important special case.  Take
\[
  \Pcal=\{P_0\},
  \qquad
  \Qcal=\M_1(\X)\setminus\{P_0\}.
\]
Corollary~\ref{cor:simple-null-one-block} gives the exact criterion.  For each \(j\), use the two one-observation scores
\[
  h_{j,+}(x)=\one_{A_j}(x)-P_0(A_j),
  \qquad
  h_{j,-}=-h_{j,+}.
\]
Their \(P_0\)-expectations are zero.  If \(Q\ne P_0\), then \(Q(A_j)\ne P_0(A_j)\) for some \(j\), and one sign has positive \(Q\)-expectation.  Thus one power-one sequential goodness-of-fit test works against every fixed alternative \(Q\ne P_0\); by Theorem~\ref{thm:change-detection}, the same certificates yield pointwise \(O_Q(\log\gamma)\) detection of every fixed change away from \(P_0\).
\end{example}

\begin{example}[Sequential two-sample testing]\label{ex:two-sample}
Suppose each round produces an independent pair \(Z_i=(X_i,Y_i)\) with law \(P\otimes Q\), and let \(\{A_j:j\ge1\}\) be countable and measure determining on \((\X,\B)\).  This is the paired, equal-rate sampling formulation; an arbitrary interleaving of two sample streams is not itself i.i.d. on one fixed pair space and requires separate bookkeeping.  On the pair space take
\[
  \Pcal_{2}:=\{R\otimes R:R\in\M_1(\X)\},
  \qquad
  \Qcal_{2}:=\{P\otimes Q:P\ne Q\}.
\]
The scores
\[
  h_{j,\pm}(x,y):=\pm\bigl(\one_{A_j}(x)-\one_{A_j}(y)\bigr)
\]
have expectation zero under every null \(R\otimes R\).  If \(P\ne Q\), some \(A_j\) satisfies \(P(A_j)\ne Q(A_j)\), and the appropriate sign has positive expectation under \(P\otimes Q\).  Hence Theorem~\ref{thm:main} gives a universal power-one sequential two-sample test, and Theorem~\ref{thm:change-detection} gives a pointwise \(O_Q(\log\gamma)\) delay bound for a change from equal to unequal marginals.
\end{example}

\begin{example}[Sequential independence testing]\label{ex:independence}
Let \((\X,\B_X)\) and \((\mathsf Y,\B_Y)\) admit countable generating \(\pi\)-systems \(\{A_j\}\) and \(\{B_k\}\) containing the respective whole spaces.  The observations are i.i.d. pairs \(Z_i=(X_i,Y_i)\).  The null class consists of all product laws, and the alternative class of all nonproduct laws \(Q\) on \(\X\times\mathsf Y\).  For two paired observations define
\[
  h_{j,k}(Z_1,Z_2)
  :=\frac12\bigl(\one_{A_j}(X_1)-\one_{A_j}(X_2)\bigr)
  \bigl(\one_{B_k}(Y_1)-\one_{B_k}(Y_2)\bigr).
\]
Every product law gives expectation zero, while under an arbitrary \(Q\),
\[
  \E_{Q^2}h_{j,k}
  =Q(A_j\times B_k)-Q_X(A_j)Q_Y(B_k).
\]
If this vanished for every \(j,k\), the \(\pi\)-\(\lambda\) theorem would imply \(Q=Q_X\otimes Q_Y\).  Hence for every nonproduct \(Q\), one of \(h_{j,k}\) or \(-h_{j,k}\) has positive expectation.  Theorem~\ref{thm:main} therefore yields a universal power-one sequential independence test, and Theorem~\ref{thm:change-detection} yields a pointwise \(O_Q(\log\gamma)\) delay bound for the onset of dependence.
\end{example}

\section{Conclusion}\label{sec:conclusion}

Power-one sequential testing and i.i.d. change detectability under fixed-horizon ARL control have the same qualitative obstruction.  A procedure exists exactly when every alternative is detected by at least one member of a countable family of finite-block certificates.  The converse requires only deterministic false-alarm horizons, while the proof supplies the stronger conclusion that one common e-detector generates the calibrated family and satisfies optional-horizon validity.  The finite-block condition can be expressed through events, randomized tests, bounded scores, bounded estimable functionals of finite degree, a reduced-time test supermartingale, or a fine-time divergent e-process.  It also has a compact geometric representation: after subdividing by rational margins, the alternative is a countable union of subfamilies whose finite-block weak-$*$ closed convex hulls are positively separated in total variation from the corresponding null hulls.

The additional results sharpen this picture.  One uniform block certificate is equivalent to uniform consistency, a uniform \(O(\log(1/\alpha))\) expected-sample-size bound, and a uniform \(O(\log\gamma)\) change-detection delay bound under fixed-horizon ARL control.  Pointwise-only testability corresponds to a genuinely countable cover, while the absence of such a cover makes both power-one testing and pointwise change detectability in this operational sense impossible.  Degree two suffices whenever testability holds on a countable discrete sample space and for every alternative outside a weakly compact null on a Polish space; simple nulls need only degree one.  The examples and the goodness-of-fit, two-sample, and independence applications show how these abstract statements reduce to elementary finite-block identities in familiar problems.

The quantitative obstruction is countable recovery of the intrinsic rate \(\Gamma\).  It is equivalent to simultaneous optimal growth, sharp expected stopping times, and sharp robust Lorden delay; countable block envelopes likewise characterize arbitrary lower rate profiles and their reciprocal delay guarantees.  Total-variation separability of the alternatives suffices, while weak lower semicontinuity of the KL projection also identifies the intrinsic rate with the raw KL benchmark.  When simultaneous recovery fails, every process has uncountably many rate deficits, but zero rates need not occur: the Lalley--Nobel product example has a positive common baseline and a strict, exactly quantified oracle gap.  Removing that baseline answers the standard-Borel countable-assembly question negatively.  Ordinary ARL alone does not support these equivalences, even with finite null means within an additive constant of the nominal scale.

Several structural questions remain.  Does every sequentially testable i.i.d. problem admit a divergent test supermartingale at every ordinary time, rather than only on a deterministic reduced-time filtration with \(t_k/k\to1\)?  Does the estimable-degree hierarchy always collapse at two, or can a composite noncompact null outside the countable-discrete regimes force higher degree?  Finally, does threshold-uniform detection delay force a sure finite-block certificate?  Beyond the i.i.d. model, controlled experiments and dependent observations will require adaptive or conditional analogues of the finite-block certificates.  A further quantitative question is how additional structure makes the countable rate-recovery constructions computationally explicit and controls their finite-sample adaptation costs.

\subsection*{Acknowledgments}

Starting from core ideas of the author, we acknowledge the use of GPT-Pro 5.6 for brainstorming further equivalent conditions, writing related work and prose, and checking proofs, though the author retains responsibility for correctness, importance and relevance of the content.

\begingroup\small\sloppy\endgroup

\end{document}